\documentclass[10pt,reqno]{amsart}

\usepackage{amsmath, amsfonts, amssymb, latexsym, amsthm}
\usepackage[pagewise]{lineno}
\usepackage{amsmath, amsfonts, amssymb, latexsym}
\usepackage[numbers,sort&compress]{natbib}
\usepackage{mathrsfs}
\usepackage{pdfcomment}
\hypersetup{hidelinks,
	colorlinks=true,
	allcolors=black,
	pdfstartview=Fit,
	breaklinks=true
}

\everymath{\displaystyle}

\newtheorem{theorem}{Theorem}
\theoremstyle{plain}

\newtheorem{lemma}[theorem]{Lemma}
\newtheorem{definition}[theorem]{Definition}

\newtheorem{proposition}[theorem]{Proposition}

\newtheorem{remark}[theorem]{Remark}

\numberwithin{equation}{section}
\numberwithin{theorem}{section}

\newcommand{\cU}{\mathcal{U}}

\newcommand{\cL}{\mathcal{L}}
\newcommand{\cJ}{\mathcal{J}}

\newcommand{\cF}{\mathcal{F}}

\newcommand{\E}{\mathbb{E}}
\newcommand{\R}{\mathbb{R}}
\newcommand{\F}{\mathbb{F}}
\newcommand{\N}{\mathbb{N}}
\newcommand{\bP}{\mathbb{P}} 

\begin{document}
\title
{Norm and time optimal control problems of stochastic heat equations}

\author{\sffamily Yuanhang Liu$^{1}$, Donghui Yang$^1$, Jie Zhong$^{2,*}$   \\
	{\sffamily\small $^1$ School of Mathematics and Statistics, Central South University, Changsha 410083, China. }\\
 {\sffamily\small $^2$ Department of Mathematics, California State University Los Angeles, Los Angeles, 90032, USA}
}
	\footnotetext[2]{Corresponding author: jiezhongmath@gmail.com }

\email{liuyuanhang97@163.com}
\email{donghyang@outlook.com}
\email{jiezhongmath@gmail.com}


\maketitle

\begin{abstract}
This paper investigates the norm and time optimal control problems for stochastic heat equations. We begin by presenting a characterization of the norm optimal  control, followed by a discussion of its properties. We then explore the equivalence between the norm optimal  control and time optimal  control, and subsequently establish the bang-bang property of the time optimal  control. These problems, to the best of our knowledge, are among the first to discuss in the stochastic case.
\end{abstract}

\pagestyle{myheadings}
\thispagestyle{plain}
\markboth{EQUIVALENCE AND THE BANG-BANG PROPERTY}{YUANHANG LIU, DONGHUI YANG, AND JIE ZHONG}

\section{Introduction}

Optimal control problems for evolution equations have garnered significant interest in recent years, with the quintessential ones being the time optimal control problems. Research on time optimal control problems for ordinary differential equations dates back to the 1950s (see, e.g., \cite{bellman1956bang}). Subsequently, in the 1960s, studies were extended to infinite dimensional cases (see, e.g., \cite{fattorini1964time,balakrishnan1965optimal}). In \cite{fattorini1964time,fattorini2005infinite}, the author derived the Pontryagin Maximum Principle for minimal time controls via the Pontryagin Maximum Principle for minimal norm controls. Furthermore, the relationship between time optimal control and norm optimal control can effectively aid in solving other  problems. For example, in \cite{carja1984minimal,carj1993minimal,gozzi1999regularity}, the authors investigated the regularity of Bellman functions associated with minimal time control problems. In \cite{tucsnak2016perturbations,yu2014approximation}, the authors explored the behavior of optimal time and optimal control when a controlled system experiences slight perturbations. In \cite{wang2013equivalence}, the authors discovered several iterative algorithms for solving time optimal control problems. Inspired by \cite{fattorini1964time,fattorini2005infinite}, the authors discussed the equivalence between the norm optimal  control and time optimal  control for internally controllable deterministic heat equations in \cite{Wang2012}. Based on this connection, they established the bang-bang property of time optimal  control with a constant control bound in \cite{Wang2015}. The bang-bang property of such a problem says, in plain language, that any optimal control reaches the boundary of the corresponding control constraint set at almost every time. This property not only is mathematically interesting, but also
has important applications. In \cite{fattorini2005infinite,wang2008null}, the authors derived the uniqueness of the optimal control from the bang-bang property. In \cite{Yang2019}, the authors discussed the bang-bang property of time optimal  control with time-varying control bounds for internally null controllable deterministic heat equations, utilizing the equivalence between norm optimal  control and time optimal  control, which is an extension work of \cite{Wang2015}. For other optimal control problems related to deterministic equations, see \cite{fattorini2005infinite,fattorini1964time,loheac2013maximum,Fattorini2011,micu2012time,phung2013observability,zuazua2006controllability,tucsnak2009observation,Yang2018}.

In practical applications, stochastic processes are often employed to model random effects, supplanting deterministic functions as the mathematical descriptions. As a result, stochastic control problems, including controllability, optimal control, and others, have drawn substantial attention from scholars. In \cite{Yang2016}, the authors provided the observability inequality of backward stochastic heat equations for measurable sets and obtained the null controllability of forward heat equations as an immediate application. In \cite{Yang2017}, the authors studied the approximate controllability for the stochastic heat equation over measurable sets and the optimal actuator location of minimum norm controls. In \cite{yan2021time}, the author investigated the time optimal control for a class of non-instantaneous impulsive Clarke subdifferential type stochastic evolution inclusions in Hilbert spaces and obtained the existence of time optimal control governed by stochastic control systems. In \cite{durga2021stochastic}, the authors discussed the time optimal control problems for time-fractional
Ginzburg–Landau equation with mixed fractional
Brownian motion. In \cite{wang2016exact}, the authors  established several sufficient conditions for a class of exact controllability problems and proved the solvability of norm optimal control problem for linear controlled forward stochastic differential equations with random coefficients. In \cite{wang2015norm}, the authors are concerned with two norm optimal control problems for different stochastic linear control systems. One is for approximately controllable systems with the natural filtration, while the other is for exactly controllable systems with a general filtration. Further, they construct the unique norm optimal control, through building up some suitable quadratic functional and making use of a variational characterization on its minimizer. For other control problems related to stochastic  equations, see \cite{Tang2009,Liu2014,Lu2011,yong1999stochastic,aoki1961stochastic,proppe1977time} and references therein.

Thus far, there has been no research on the equivalence between the  norm optimal control and time optimal  control and the bang-bang property of time optimal  control for stochastic parabolic equations. This paper presents the first attempt to address these optimal control problems for stochastic systems. Compared to the deterministic case, stochastic terms arise when studying the properties of the norm optimal  control problem, which complicates further study of the time optimal  control problem. To overcome this difficulty, we adopt relevant techniques from \cite{lv1} to establish the equivalence between norm optimal  control and time optimal  control and to obtain the bang-bang property of time optimal  control. On the other hand, it is important to recognize that we cannot employ the time change technique and treat the backward and forward equations in the same manner as in the deterministic case. This is because the stochastic system requires adaptedness, which cannot be disregarded in calculations. The adaptedness of stochastic processes has emerged as a crucial hindrance in exploring the equivalence of time optimal control and norm optimal control problems. Notably, the optimal control in the deterministic case can be constructed in the space $L^\infty(0,T;U)$ to show such an equivalence. However, for stochastic equations, the issue has yet to be resolved with control in the similar space. As a result, this paper aims to delve into stochastic equations with control in the space $L^2_{\mathbb{F}}(0,T;L^2(\Omega;U))$.

Before we state the problems and main results, let us introduce necessary notations.

Let $D$ be an open and bounded domain with smooth boundary in
$\R^l$, $l\in \N^+$. $\R^+:=(0,+\infty)$ and $G$ is nonempty and open subset of $D$. Denote by $T_n\uparrow T$ for $T_n \rightarrow T$ as $n\rightarrow \infty$ and $T_n \leq T_{n+1}$ for all $n \in \N$, and $T_n\downarrow T$ for $T_n \rightarrow T$ as $n\rightarrow \infty$ and $T_n \geq T_{n+1}$ for all $n \in \N$.

Let $(\Omega,\mathcal{F},\left\{\mathcal{F}_t\right\}_{t\geq 0}, \bP)$ is a fixed complete filtered probability space, on which a one dimensional standard Brownian motion $\{W(t)\}_{t\geq0}$ is defined, and $\left\{\mathcal{F}_t\right\}_{t\geq 0}$ is the corresponding natural filtration, augmented by all the $\bP$-null sets in $\mathcal{F}$. We denote by $\mathbb{F}$ the progressive $\sigma$-field w.r.t. $\left\{\mathcal{F}_t\right\}_{t\geq 0}$.

Given a Hilbert space $H$, Fix $t \ge 0, p \in [1,\infty)$, we denote by $L^p_{\mathcal{F}_t}(\Omega;H)$  the Banach space consisting of all $H$-valued,  $\mathcal{F}_t$ measurable random
variables $X(t)$ endowed with the norm
$$
\|X(t)\|_{L^p_{\mathcal{F}_t}(\Omega;H)}=\bigg(\mathbb{E}\|X(t)\|^p_{H}\bigg)^{\frac{1}{p}}.
$$
Denote by $L^p_{\mathbb{F}}(0,T;L^q(\Omega;H))$, $p,q\in[1,\infty)$, the Banach space consisting of all $H$-valued, $\left\{\mathcal{F}_t\right\}_{t\geq 0}$-adapted processes $X$ endowed with the norm
$$
\|X(\cdot)\|_{L^p_{\mathbb{F}}(0,T;L^q(\Omega;H))}=\bigg(\int_0^T(\mathbb{E}\|X(t)\|^q_H)^{\frac{p}{q}}dt\bigg)^{\frac{1}{p}}.
$$
Denote by $L^p_{\mathbb{F}}(\Omega; L^q(0,T;H))$, $p,q\in[1,\infty)$, the Banach space consisting of all $H$-valued, $\left\{\mathcal{F}_t\right\}_{t\geq 0}$-adapted processes $X$ endowed with the norm
$$
\|X(\cdot)\|_{L^p_{\mathbb{F}}(\Omega; L^q(0,T;H))}=\bigg[\mathbb{E}\bigg(\int_0^T\|X(t)\|^q_H dt\bigg)^{\frac{p}{q}}\bigg]^{\frac{1}{p}}.
$$
Denote by $L^\infty_{\mathbb{F}}(0,T;\R)$, the Banach space consisting of all $\R$-valued, $\left\{\mathcal{F}_t\right\}_{t\geq 0}$-adapted bounded processes, with the essential supremum norm.\\
Denote by $L^q_{\mathbb{F}}(\Omega;C([0,T];H))$, $q\in[1,\infty)$, the Banach space consisting of all $H$-valued, $\left\{\mathcal{F}_t\right\}_{t\geq 0}$-adapted continuous processes $X$ endowed with the norm
$$
\|X(\cdot)\|_{L^q_{\mathbb{F}}(\Omega;C([0,T];H))}=\bigg(\mathbb{E}\|X(\cdot)\|^q_{C([0;T];H)}\bigg)^{\frac{1}{q}}.
$$

Throughout this paper, we denote by $\langle\cdot,\cdot\rangle$ the scalar product in $L^2(D)$ and denote
by $\|\cdot\|$ the norm induced by $\langle\cdot,\cdot\rangle$. We denote by $|\cdot|$ the
Lebesgue measure on $L^2(D)$. In the sequel, we shall simply denote $L^p_{\mathbb{F}}(0,T;L^p(\Omega;H))\equiv L^p_{\mathbb{F}}(\Omega; L^p(0,T;H))$ by $L^p_{\mathbb{F}}(0,T;H)$ with $p\in[1,\infty)$.

The rest of the paper is organized as follows. In section 2, the main problem is formulated and the main results Theorem \ref{th-equivalence-TM-N} and Theorem \ref{th1} are stated. In section 3, some auxiliary results to be used later are presented. In section 4, we presents some properties on norm optimal control problem ${\bf(NP)}_{y_0}^T$ (see (\ref{problem:N})) which is to be formulated later. These are beneficial to the proofs of the main results in section 5.

\section{Problem formulation and main results}

The system, that we consider in this paper, is described by the following controlled stochastic heat
equation:
\begin{equation}
\label{model2}
\left\{
\begin{array}{ll}
dy=\triangle ydt+\chi_Gudt+aydW(t), \quad\ &\mathrm{in} \quad\ D\times\R^+,\\[3mm]
y=0, \quad\ &\mathrm{on} \quad\ \partial D\times\R^+,\\[3mm]
y(0)=y_0,\quad\ &\mathrm{in} \quad\  D,
\end{array}
\right.
\end{equation}
where  $y_0 \in L^2_{\mathcal{F}_0}(\Omega;L^2(D))$ is the initial state, $a\in L^\infty_{\F}(0,T;\R)$, $y$ is an $L^2(D)$-valued state variable, $u \in L_{\F}^2(\R^+; L^2(D))$. It is well known (see, \cite{lv1}) that for any $y_0\in L^2_{\mathcal{F}_0}(\Omega;L^2(D))$, the system (\ref{model2}) admits a unique solution $y$ in the space of $L^2_{\mathbb{F}}(\Omega;C([0,T];L^2(D))) \times L^2_\mathbb{F}(0,T;H_0^1(D))$. Moreover, there is a positive constant $C$ such that
\begin{equation}
\label{ineq-y}
\|y\|_{L^2_{\mathbb{F}}(\Omega;C([0,T];L^2(D))) \cap L^2_\mathbb{F}(0,T;H_0^1(D))}
\leq C\big(\|y_0\|_{L^2_{\mathcal{F}_0}(\Omega;L^2(D))}+\|u\|_{L_{\F}^2(0,T;L^2(G))}\big).
\end{equation}

Throughout the paper, we assume that $y_0\neq 0$ and we denote by $y(\cdot; y_0,u)$ the solution of system (\ref{model2}).

Given $N>0, y_{0}\in L^2_{\mathcal{F}_0}(\Omega;L^2(D))$, we consider following time optimal control problem ${\bf(TP)}_{y_0}^N$:
\begin{equation}
\label{problem:TN}
T (N, y_0) := \inf_{u \in \cU_N}\left\{t \in \R^+: y(t; y_0,u) = 0\right\},
\end{equation}
where $y(\cdot; y_0, u)$ is the solution of system (\ref{model2}) and
$$
\cU_N:=\left\{ u\in L_{\F}^2(\R^+;L^2(D)): \|u\|_{L^2_{\cF_t}(\Omega;L^2(D))}\leq N,\,\, t\in \R^+ \, a.e. \right\}.
$$

In problem ${\bf(TP)}_{y_0}^N$, a tetrad $(0,y_0, t, u)$ is called admissible, if $t \in \R^+$,\,$u \in \cU_N$, and $y(t ; y_0, u)=0$ ; A tetrad $(0,y_0,T (N, y_0),u^*)$ is called optimal, if $T (N, y_0)\in \R^+,\,u^* \in \cU_N \,\,\mathrm{and}\,\,y(T (N, y_0);y_0,u^*)=0;$ When $(0,y_0,T (N, y_0),u^*)$ is an optimal tetrad, $T (N, y_0)$ and $u^*$ are called the optimal time and a time optimal  control, respectively.

\begin{definition}
Time optimal  control problem ${\bf(TP)}_{y_0}^N$ has the bang-bang property if any time optimal  control $u^*$ verifies  $\|\chi_Gu^*\|_{L_{\F}^2(0,T (N, y_0); L^2(D))}=N$ and\\ $\|\chi_Gu^*(t)\|_{L^2_{\cF_t}(\Omega;L^2(D))} \neq0$ for $t \in(0,T (N, y_0))$ a.e.

\end{definition}

In order to derive the bang-bang property of problem and ${\bf(TP)}_{y_0}^N$, we consider a norm optimal control problem ${\bf(NP)}_{y_0}^T$ of the following:
\begin{equation}
\label{problem:N}
N (T, y_0) :=
\inf
\left\{
\|u\|_{L^2_{\F}(0,T;L^2(D))}:
y(T; y_0,u) = 0,\,\, \bP\text{ - a.s.}
\right\},
\end{equation}
where $y(\cdot; y_0, u)$ is the solution of the following system:
\begin{equation}
\label{model3}
\left\{
\begin{array}{ll}
dy=\triangle ydt+\chi_Gudt+aydW(t), \quad\ &\mathrm{in} \quad\ D\times(0,T),\\[3mm]
y=0, \quad\ &\mathrm{on} \quad\ \partial D\times(0,T),\\[3mm]
y(0)=y_0,\quad\ &\mathrm{in} \quad\  D.
\end{array}
\right.
\end{equation}

In problem ${\bf(NP)}_{y_0}^T$, a control $u \in L^2_{\F}(0,T;L^2(D))$ is called admissible, if $y(T ;  y_0, u) \\= 0$ and a control $u_* \in L^2_{\F}(0,T;L^2(D))$ is called optimal, if it is admissible and $\|u_*\|_{L_{\F}^2(0,T;L^2(D))}= N(T, y_0)$, which is the optimal norm.

\begin{definition}
Norm optimal  control problem ${\bf(NP)}_{y_0}^T$ has the bang-bang property if any norm optimal  control $u_*$ verifies that $\|\chi_Gu_*\|_{L_{\F}^2(0,T; L^2(D))}=N(T,y_0)$ and $\|\chi_Gu_*(t)\|_{L^2_{\cF_t}(\Omega;L^2(D))}\neq0$ for $t \in(0,T )$ a.e.
\end{definition}

Another goal of this paper, motivated by \cite{Wang2012,Wang2018}, is to discuss the equivalence between the time optimal control problem ${\bf(TP)}_{y_0}^N$ and the norm optimal control problem ${\bf(NP)}_{y_0}^T$.
\begin{definition}
\label{definition-equivalence-TM-N}
Let $N> 0$ and $T\in\R^+$, the time optimal control problem ${\bf(TP)}_{y_0}^N$ and the norm optimal control problem ${\bf(NP)}_{y_0}^T$ are
said to be equivalent if the following three conditions hold:
\begin{itemize}
  \item [$(a)$] Both problem ${\bf(TP)}_{y_0}^N$ and problem ${\bf(NP)}_{y_0}^T$ have optimal controls;
  \item [$(b)$] The restriction of each optimal control to problem ${\bf(TP)}_{y_0}^N$ over $(0, T )$ is an optimal control to problem ${\bf(NP)}_{y_0}^T$;
  \item [$(c)$] The zero extension of each optimal control to problem ${\bf(NP)}_{y_0}^T$ over $\R^+$ is an optimal control to problem ${\bf(TP)}_{y_0}^N$.
\end{itemize}

\end{definition}

Denote
\begin{equation}
\label{equivalence}
\Lambda:=\bigg\{ \big(N(T,y_0),T\big)\bigg|T\in \R^+\bigg\}.
\end{equation}

The main results of this paper are stated as follows:\\

\begin{theorem}
\label{th-equivalence-TM-N}
 When $(N, T ) \in \Lambda$ , problems ${\bf(TP)}_{y_0}^N$ and problems ${\bf(NP)}_{y_0}^T$ are equivalent.
\end{theorem}
\begin{remark}
In Chapter 5 of \cite{Wang2018}, the authors demonstrated the equivalence and inequivalence of time optimal control and norm optimal control problems for deterministic evolution equations with control $u\in L^\infty(0,T;L^2(D))$. However, in the stochastic equations, similar results have yet to be obtained with control $u\in L^\infty_{\mathbb{F}}(0,T;L^2(\Omega;L^2(D)))$, as the adaptedness of the stochastic process must be taken into account with the similar methods. As a result, we only obtain the equivalence of time optimal control and norm optimal control with $u\in L_{\F}^2(0,T;L^2(D))$.
\end{remark}

\begin{theorem}
\label{th1}
Let $y_0 \in L_{\cF_0}^2(\Omega;L^2(D))\setminus \left\{ 0 \right\}$. Then, the time optimal control problem ${\bf(TP)}_{y_0}^N$ admits a solution, i.e., there exists $u^*\in L_{\F}^2(0,T(N,y_0);L^2(D))$ such that $y(T (N, y_0); y_0, u^*) = 0$, and the time optimal  control $u^*$ satisfies the bang-bang property.
\end{theorem}

\section{Auxiliary conclusions}
In this section, we give some auxiliary results that will be used later. Let us introduce the following backward stochastic heat equation:
\begin{equation}
\label{BSDE}
\left\{
\begin{array}{ll}
dz=-\triangle zdt-aZdt+ZdW(t), \quad\ &\mathrm{in} \quad\ D\times(0,T),\\[3mm]
z=0, \quad\ &\mathrm{on} \quad\ \partial D\times(0,T),\\[3mm]
z(T)=\eta,\quad\ &\mathrm{in} \quad\  D.
\end{array}
\right.
\end{equation}
It is well known (see, (\cite{lv1}, Theorem 4.10.)) that for any $\eta\in L^2_{\mathcal{F}_T}(\Omega;L^2(D))$, the system (\ref{BSDE}) admits a unique solution $(z,Z)$ in the space of $L^2_{\mathbb{F}}(\Omega;C([0,T];L^2(D)))\\ \times L^2_\mathbb{F}(0,T;L^2(D))$.
\begin{lemma}
\label{lemma:observe ine}
The system (\ref{BSDE}) is exactly observable, i.e., there exists a constant $C=C(T,D, G) > 0$, such that for all $\eta\in L_{\cF_T}^2(\Omega;L^2(D))$,
\begin{equation}
\label{observe ine}
\begin{array}{ll}
\|z(0)\|^2_{L_{\cF_0}^2(\Omega;L^2(D))}\leq C \|\chi_Gz\|^2_{L_{\F}^2(0,T;L^2(D))}.
\end{array}
\end{equation}
Here and in what follows, we denote by $C$ a constant although it may have different values in different contexts.

\end{lemma}
\begin{proof}
From observability inequality (1.2) in \cite{Yang2016}, it implies desired estimate (\ref{observe ine}).

\end{proof}

\begin{remark}
\label{remark1}
The null controllability for system (\ref{model2}) or (\ref{model3}) is equivalent to the observability inequality (\ref{observe ine}) for system (\ref{BSDE}); see, for instance \cite{lv1,Yang2016}.

\end{remark}

\begin{lemma}
\label{lemma:null controllable}
The system (\ref{model2}) or (\ref{model3}) is null controllable at time $T$. That is, for any $y_0\in L^2_{\mathcal{F}_0}(\Omega;L^2(D))$, there exists a control $u \in L^2_{\mathbb{F}}(0,T;L^2(D))$ such that $y(T;y_0,u)=0$, in $D$, $\bP$-a.s.. Moreover, the control $u$ satisfies the following estimate
$$
\|u\|^2_{L^2_{\mathbb{F}}(0,T;L^2(D))}\leq C\|y_0\|^2_{L^2_{\mathcal{F}_0}(\Omega;L^2(D))}.
$$

\end{lemma}

Here and what follows, we simply set $z(\cdot;\eta) = z(\cdot;T,\eta)$ for the solution of the adjoint equation \eqref{BSDE} with the
terminal condition $z(T) = \eta$.

Next, for any $T \in \R^+$, set
$
X=\left\{z(\cdot;\eta):\eta\in L^2_{\mathcal{F}_T}(\Omega;L^2(D))  \right\}.
$
Define $\|\cdot\|_{X}: X \to \R$ by
\begin{equation}
  \label{norm:X}
\|z\|_X = \| z\|_{L^2_\mathbb{F}(0,T;L^2(D))} = \left( \int_0^T \E \|\chi_G z\|^2_{L^2(D)} dt\right)^{1/2}.
\end{equation}
It follows from the observability inequality \eqref{observe ine} that $\|\cdot\|_X$ is indeed a norm on space $X$. We denote by $Y$ the completion of the space $X$ under the norm $\|\cdot\|_X$. The following proposition provides us a description of $Y$.

\begin{lemma}
  \label{lemma:Y}
  Under an isomorphism, any element of\, $Y$ can be expressed
  as a process $\varphi(\cdot;T,\eta) \in L^2_{\mathbb{F}}(\Omega;C([0,T];L^2(D)))$, which
  satisfies
\begin{equation}
  \label{BSDE2}
  \left\{
\begin{array}{ll}
  d\varphi=-\triangle \varphi dt-aZdt+ZdW(t), \quad\ &\mathrm{in} \quad\ D\times(0,T),\\[3mm]
\varphi=0, \quad\ &\mathrm{on} \quad\ \partial D\times(0,T),\\[3mm]
\varphi(T)=\eta,\quad\ &\mathrm{in} \quad\  D,
\end{array}
\right.
\end{equation}
 for some $Z\in L^2_{\mathbb{F}}(0,T;L^2(D))$, $\bP$-a.s.
 Moreover, $\chi_G\varphi(\cdot;T,\eta) =
\lim_{n\to\infty}\chi_G z(\cdot;\eta_n)$
for some sequence
$\{\eta_n\}\subseteq L^2_{\mathcal{F}_T}(\Omega;L^2(D))$.
\end{lemma}

The proof is similar to the proof of Lemma 3.4 in \cite{Yang2016}. Therefore, we omit the details.

\begin{remark}
  \label{remark:Y}
  The element $\varphi$ in $Y$ is not necessarily in the space
  of $L^2_{\F}(0,T;L^2(D))$, but $\chi_G\varphi(\cdot;T,\eta)\in
L^2_{\F}(0,T;L^2(D))$. Also, because of the isomorphism, we write \begin{equation}
\label{eq:Y-norm}
\|\varphi\|_Y=\|\chi_G \varphi(\cdot;T,\eta)\|_{L^2_{\F}(0,T;L^2(D))}= \left( \int_0^T \E \|\chi_G\varphi(t;T,\eta)\|^2_{L^2(D)} dt\right)^\frac{1}{2}.
\end{equation}
\end{remark}

\begin{remark}
  \label{remark:not zero}
when $\xi\in Y\setminus \{0\}$, it holds that $\|\chi_G\xi\|_{L^2_{\cF_t}(\Omega;L^2(D))}\neq0$ for each $t \in[0, T )$.
Indeed, from Lemma \ref{lemma:Y}, there is a function $\varphi \in L^2_\mathbb{F}(\Omega;C([0,T];L^2(D)))$ with $\chi_G \varphi\in L^2_\mathbb{F}(0,T;L^2(D))$ such that $\xi=\chi_G \varphi$. This follows that $\varphi\neq0$. By (\ref{observe ine2}) in Lemma \ref{lemma:observe ine2} (it will be given later), it follows that $\chi_G \varphi\neq0$ in $L^2_{\cF_t}(\Omega;L^2(D))$ for each $t \in [0, T )$.

\end{remark}

In a word, from Lemma \ref{lemma:Y} and Remark \ref{remark:Y}, the space $Y$ can be described as follows:
\begin{equation}
\label{space-Y}
Y=
\left\{
\begin{array}{lll}
&\text{Solving the system (\ref{BSDE2}) such that}~    \chi_G \varphi\in L^2_{\F}(0,T;L^2(D)).\\[3mm]
\varphi:&\text{Moreover, for any}~\eta_n\in L^2_{\cF_T}(\Omega;L^2(D)), \\[3mm]
&\chi_G \varphi = \lim_{n\to \infty} \chi_G z(\cdot;\eta_n)
\end{array}
\right \},
\end{equation}
endowed with the norms (\ref{eq:Y-norm}).

Similar to the Lemma \ref{lemma:observe ine}, we also have the following observability inequality for system (\ref{BSDE2}):
\begin{lemma}
\label{lemma:observe ine2}
The system (\ref{BSDE2}) is exactly observable, i.e., there exists a constant $C=C(T,D, G) > 0$, such that for all $\eta\in L_{\cF_T}^2(\Omega;L^2(D))$,
\begin{equation}
\label{observe ine2}
\begin{array}{ll}
\|\varphi(0;T,\eta)\|^2_{L_{\cF_0}^2(\Omega;L^2(D))}\leq C \|\chi_G\varphi(\cdot;T,\eta)\|^2_{L_{\F}^2(0,T;L^2(D))}.
\end{array}
\end{equation}

\end{lemma}

Next,  introduce the attainable subspace
$
A=\bigg\{y(T;0,u):u\in L^2_{\F}(0,T;L^2(D))\bigg\},
$
endowed with the norms
\begin{equation}
\label{eq:A-norm}
\|y_T\|_A:=\inf\bigg\{ \|u\|_{L^2_{\F}(0,T;L^2(D))}: y(T;0,u)=y_T\bigg\},\, y_T\in A.
\end{equation}
We have the following result about the space $A$, which will be used later in this paper.
\begin{lemma}
\label{lemma:A}
Let $T>0$, there is a linear operator $g$ from $A$ to $L^2_{\F}(0,T;L^2(D))$ such that $g$ preserves the norms.
\end{lemma}
\begin{proof}
For any $\eta\in L_{\cF_T}^2(\Omega;L^2(D))$ and $v\in L^2_{\F}(0,T;L^2(D))$, by system (\ref{model3}), system (\ref{BSDE2}) and It\^{o}'s formula, we have
\begin{equation}
\label{eq:itoA}
\E \left\langle y(T;0,v), \eta\right\rangle =\E\int_0^T \left\langle v(t), \chi_G \varphi(t;T,\eta)\right\rangle dt.
\end{equation}
Let $y_T \in A$. Then $y_T = y(T ; 0,\hat{u})$ for some $\hat{u} \in L^2_{\F}(0,T;L^2(D))$. Define $\cL_{y_T}:L_{\F}^2(0,T;L^2(D))\rightarrow \R$ by setting
\begin{equation}
\label{eq:L-yT}
\cL_{y_T}(\chi_G \varphi(\cdot;T,\eta))=\E\int_0^T \left\langle \hat{u}(t), \chi_G\varphi(t;T,\eta)\right\rangle dt,~
\text{for each}~ \eta\in L_{\cF_T}^2(\Omega;L^2(D)).
\end{equation}
From (\ref{eq:itoA}) and (\ref{eq:L-yT}), one can easily check that $\cL_{y_T}$ is well-defined and linear. Meanwhile, on one hand, using the H\"{o}lder inequality to the right side of equality (\ref{eq:L-yT}), we see that $\cL_{y_T}$ is bounded. On the other hand, since the dual of the space $L_{\F}^2(0,T;L^2(D))$ equals itself, see Theorem 2.73 in \cite{lv1}, $\cL_{y_T}$ is a linear bounded functional on $L_{\F}^2(0,T;L^2(D))$. i.e.,
\begin{equation}
\label{eq:L-yT-bound}
\cL_{y_T}\in L_{\F}^2(0,T;L^2(D)).
\end{equation}
Now defined $g:A \rightarrow L_{\F}^2(0,T;L^2(D))$ by setting
\begin{equation}
\label{eq:g}
g(y_T)=\cL_{y_T}\,\, \text{for each}~ y_T\in A.
\end{equation}
It is clear that $g$ is linear.

We now prove that $g$ preserves the norms. i.e.,
\begin{equation}
\label{eq:g-preserve}
\|g(y_T)\|_{L_{\F}^2(0,T;L^2(D))}=\|y_T\|_A\,\, \text{for each}~ y_T\in A.
\end{equation}
Let $y_T\in A$. Arbitrarily take a $\hat{u} \in L_{\F}^2(0,T;L^2(D))$ such that $y_T = y(T ; 0, \hat{u})$.\\
From equality (\ref{eq:L-yT}) and (\ref{eq:L-yT-bound}), it follows that
\begin{equation}
\label{lemma:A-eq1}
\left\langle \cL_{y_T},\xi \right\rangle_{L_{\F}^2(0,T;L^2(D))}=\E \int_0^T\left\langle \hat{u}(t),\xi(t) \right\rangle dt    \,\, \text{for each}~ \xi\in L_{\F}^2(0,T;L^2(D)),
\end{equation}
which implies
$$
\|\cL_{y_T}\|_{L_{\F}^2(0,T;L^2(D))}\leq \|\hat{u}\|_{L_{\F}^2(0,T;L^2(D))}.
$$
Then together with (\ref{eq:A-norm}), it leads to
\begin{equation}
\label{ineq:g-preserve-one side}
\|\cL_{y_T}\|_{L_{\F}^2(0,T;L^2(D))}\leq \|y_T\|_A.
\end{equation}
Conversely, for any $\hat{v}\in L_{\F}^2(0,T;L^2(D))$, from (\ref{lemma:A-eq1}), we know
\begin{equation}
\label{lemma:A-ineq1}
\bigg|\E \int_0^T\left\langle \hat{v}(t),\xi(t) \right\rangle dt\bigg|\leq \|\cL_{y_T}\|_{L_{\F}^2(0,T;L^2(D))}\|\xi\|_{L_{\F}^2(0,T;L^2(D))}.
\end{equation}
Defined $h: L_{\F}^2(0,T;L^2(D)) \rightarrow \R$ by setting
\begin{equation}
\label{eq:h}
h(\xi)=\E \int_0^T\left\langle \hat{v}(t),\xi(t) \right\rangle dt\,\, \text{for each}~ \xi \in L_{\F}^2(0,T;L^2(D)).
\end{equation}
By (\ref{lemma:A-ineq1}) and (\ref{eq:h}), it follows that $h$ is a linear bounded functional on $L_{\F}^2(0,T;L^2(D))$, which shows that
\begin{equation}
\label{ineq:h-bound}
\|h\|_{L_{\F}^2(0,T;L^2(D))}\leq \|\cL_{y_T}\|_{L_{\F}^2(0,T;L^2(D))}.
\end{equation}
According to the Riesz-type representation theorem for general stochastic processes (see Theorem 2.55 in \cite{lv1}), $\exists\,v\in L_{\F}^2(0,T;L^2(D))$ such that $\forall \,\xi\in L_{\F}^2(0,T;L^2(D))$, we have
$$
h(\xi)=\E \int_0^T\left\langle v(t),\xi(t) \right\rangle dt.
$$
Moreover,
\begin{equation}
\label{eq:h-riesz}
\|h\|_{L_{\F}^2(0,T;L^2(D))}= \|v\|_{L_{\F}^2(0,T;L^2(D))}.
\end{equation}
Let $y_T\in A$. We fix the $v\in L_{\F}^2(0,T;L^2(D))$ such that $y(T;0,v)=y_T$.
By (\ref{eq:A-norm}), we have
$$
\|y_T\|_A\leq \|v\|_{L_{\F}^2(0,T;L^2(D))},
$$
which combines with (\ref{ineq:h-bound}) and (\ref{eq:h-riesz}), leads to
\begin{equation}
\label{ineq:g-preserve-another side}
\|\cL_{y_T}\|_{L_{\F}^2(0,T;L^2(D))}\geq \|y_T\|_A.
\end{equation}
To sum up, from (\ref{ineq:g-preserve-one side}) and (\ref{ineq:g-preserve-another side}), we obtain
\begin{equation}
\label{ineq:g-preserve-version}
\|\cL_{y_T}\|_{L_{\F}^2(0,T;L^2(D))}= \|y_T\|_A.
\end{equation}
Now (\ref{eq:g-preserve}) follows from the definition of $g$ in (\ref{eq:g}) and (\ref{ineq:g-preserve-version}). This completes the proof.
\end{proof}

\section{The properties of the norm optimal  control problem}
In this section, we presents some properties of the norm optimal  control problem ${\bf(NP)}_{y_0}^T$. These properties will be used in the proof of our main results.

\subsection{Characterizations of the norm optimal  control problem}

\begin{lemma}
\label{lemma:N-characteristic}
For all $T>0$, $y_0\in L_{\cF_0}^2(\Omega;L^2(D))\setminus\{0\}$. The optimal norm $N(T , y_0)$ of problem ${\bf(NP)}_{y_0}^T$ can be characterized as
\begin{equation}
\label{eq:N-characteristic}
N(T , y_0)=\sup_{\eta\in L_{\cF_T}^2(\Omega;L^2(D))\setminus\{0\}} \dfrac{\E\left\langle y_0,\varphi(0;T,\eta) \right\rangle}{\|\chi_G\varphi(\cdot;T,\eta)\|_{L_{\F}^2(0,T;L^2(D))}}.
\end{equation}

\end{lemma}

\begin{proof}
Let $T>0$, $y_0\in L_{\cF_0}^2(\Omega;L^2(D))\setminus\{0\}$. Write $y_T:= -y(T;y_0,0)$. By linearity, one can easily check that
$
y(T;y_0,u)=0\Leftrightarrow y(T;0,u)=y_T.
$
According to the Lemma \ref{lemma:null controllable}, the system (\ref{model2}) or (\ref{model3}) is null controllable.
Hence, the above expression leads to $y_T\in A$. Then from (\ref{problem:N}) and (\ref{eq:A-norm}), we have
\begin{equation}
\label{lemma:N-characteristic-eq1}
N(T , y_0)=\|y_T\|_A.
\end{equation}
Let $\hat{u} \in L_{\F}^2(0,T;L^2(D))$ be such that $y(T ; 0, \hat{u}) = y_T$. Taking $v=\hat{u}$ in (\ref{eq:itoA}) and together with (\ref{eq:L-yT}), it follows that
\begin{equation}
\label{lemma:N-characteristic-eq2}
\cL_{y_T}(\chi_G\varphi(\cdot;T,\eta))=\E\left\langle y_T,\eta \right\rangle \,\,\text{for each}~ \eta\in L_{\cF_T}^2(\Omega;L^2(D)).
\end{equation}
Since $\cL_{y_T}$ is a linear bounded functional on $L_{\F}^2(0,T;L^2(D))$, as well as from Lemma \ref{lemma:observe ine2} we have $\chi_G\varphi(t;T,\eta)\neq 0$ when $\eta\in L_{\cF_T}^2(\Omega;L^2(D))\setminus\{0\}$ and $t \in [0, T )$ a.s.. Hence, from (\ref{lemma:N-characteristic-eq2})
\begin{equation}
\label{lemma:N-characteristic-eq3}
\|\cL_{y_T}\|_{L_{\F}^2(0,T;L^2(D))}=\sup_{\eta\in L_{\cF_T}^2(\Omega;L^2(D))\setminus\{0\}} \dfrac{\E\left\langle y(T;y_0,0),\eta \right\rangle}{\|\chi_G\varphi(\cdot;T,\eta)\|_{L_{\F}^2(0,T;L^2(D))}}.
\end{equation}
Consider the following stochastic heat equation, i.e., equation (\ref{model3}) with $u=0$:
\begin{equation}
\label{model4}
\left\{
\begin{array}{ll}
dy=\triangle ydt+aydW(t), \quad\ &\mathrm{in} \quad\ D\times(0,T),\\[3mm]
y=0, \quad\ &\mathrm{on} \quad\ \partial D\times(0,T),\\[3mm]
y(0)=y_0,\quad\ &\mathrm{in} \quad\  D.
\end{array}
\right.
\end{equation}
By system (\ref{model4}), system (\ref{BSDE2}) and It\^{o}'s formula, we have
\begin{equation}
\label{eq:without u}
\E\left\langle y(T;y_0,0),\eta \right\rangle=\E\left\langle y_0,\varphi(0;T,\eta) \right\rangle.
\end{equation}
Combining the above equality with (\ref{lemma:N-characteristic-eq3}), it follows that
$$
\|\cL_{y_T}\|_{L_{\F}^2(0,T;L^2(D))}=\sup_{\eta\in L_{\cF_T}^2(\Omega;L^2(D))\setminus\{0\}} \dfrac{\E\left\langle y_0,\varphi(0;T,\eta) \right\rangle}{\|\chi_G\varphi(\cdot;T,\eta)\|_{L_{\F}^2(0,T;L^2(D))}}.
$$
This, together with (\ref{ineq:g-preserve-version}) and (\ref{lemma:N-characteristic-eq1}), the conclusion (\ref{eq:N-characteristic}) is true. The proof is completed.
\end{proof}

Here and what follows, we simply set $\varphi(\cdot;\eta) = \varphi(\cdot;T,\eta)$ for the solution of the equation \eqref{BSDE2} with the terminal condition $\varphi(T) = \eta$.

We define a functional $\cJ$ on $Y$:
\begin{equation}
\label{functional:J}
\begin{array}{lll}
\cJ(\varphi)&=&\dfrac{1}{2}\|\chi_G \varphi(\cdot;\eta)\|^2_{L_{\F}^2(0,T;L^2(D))}+\E\left\langle y_0, \varphi(0;\eta) \right\rangle\\[3mm]
&=&\dfrac{1}{2}\int_0^T \E\|\chi_G \varphi\|^2_{L^2(D)}dt+\E\left\langle y_0,\varphi(0;\eta) \right\rangle.
\end{array}
\end{equation}

Denote the variational problem ${\bf(VP)}_{y_0}^T$ as follows:
\begin{equation}
\label{problem:V}
V := \inf_{\varphi\in Y}\cJ(\varphi).
\end{equation}
Next, we prove the existence of the minimizer of $\cJ$ defined in (\ref{functional:J}).
\begin{lemma}
\label{lemma:J}
Suppose $y_0\in  L_{\cF_0}^2(\Omega;L^2(D))$, the functional $\cJ$ admits a minimizer on $Y$.
\end{lemma}
\begin{proof}
One can easily show that $\cJ(\cdot)$ is strict convex and continuous. Next, we prove the coercivity, let $\left\{\varphi_n\right\}\subseteq Y$ such that $\|\varphi_n\|_Y\rightarrow\infty$ as $n\rightarrow\infty$ and let
$
\tilde{\varphi_n}=\dfrac{\varphi_n}{\|\varphi_n\|_Y},
$
so that $\|\tilde{\varphi_n}\|_Y=1$.
Then according to (\ref{eq:Y-norm}), we have
$$
\begin{array}{lll}
\dfrac{\cJ(\varphi_n)}{\|\varphi_n\|_Y}&=&\dfrac{1}{2}\|\varphi_n\|_Y\int_0^T \E \|\chi_G \tilde{\varphi_n}\|_{L^2(D)}^2dt +\E\left\langle y_0,\tilde{\varphi_n}(0)\right\rangle \\[3mm]
&=&\dfrac{1}{2}\|\varphi_n\|_Y\|\tilde{\varphi_n}\|^2_Y+\E\left\langle y_0, \tilde{\varphi_n}(0)\right\rangle \\[3mm]
&=&\dfrac{1}{2}\|\varphi_n\|_Y+\E\left\langle y_0, \tilde{\varphi_n}(0)\right\rangle .
\end{array}
$$
Noting that
$
\big|\,\E\left\langle y_0, \tilde{\varphi_n}(0)\right\rangle\big|
\leq C\|y_0\|_{L_{\cF_0}^2(\Omega;L^2(D))},
$
which along with estimate (\ref{observe ine2}) in Lemma \ref{lemma:observe ine2}, we have
$
\cJ(\varphi_n)\rightarrow\infty \,\, \mathrm{as} \,\,\|\varphi_n\|_Y\rightarrow\infty.
$
This shows the coercivity of $\cJ(\cdot)$.

To sum up, we showed that $\cJ(\cdot)$ is strict convex, continuous, and coercive, and thus the minimizer of
$\cJ(\cdot)$ exists.
\end{proof}

\begin{remark}
\label{remark:zero}
If we suppose $y_0\in  L_{\cF_0}^2(\Omega;L^2(D))\setminus\{0\}$, then zero is not a minimizer of the functional $\cJ$, see  Lemma 2.2 in \cite{Wang2015} for more details.

\end{remark}

The studies on ${\bf(NP)}_{y_0}^T$ are closely related to the variational problem ${\bf(VP)}_{y_0}^T$ and the following conclusion describes the relation between ${\bf(NP)}_{y_0}^T$ and ${\bf(VP)}_{y_0}^T$.
\begin{proposition}
\label{relation:N-V}
For all $T>0$, and $y_0\in  L_{\cF_0}^2(\Omega;L^2(D))\setminus\{0\}$, then
\begin{equation}
\label{eq:relation-N-V}
V=-\dfrac{1}{2}N(T,y_0)^2.
\end{equation}

\end{proposition}

\begin{proof}
We first prove that
\begin{equation}
\label{relation:N-V-one side}
V\geq-\dfrac{1}{2}N(T,y_0)^2 \quad\ \text{for all}~\,\, T>0 \,\, \text{and}~\,\, y_0\in  L_{\cF_0}^2(\Omega;L^2(D))\setminus\{0\}.
\end{equation}
From the Lemma \ref{lemma:observe ine2}, it follows that $\chi_G\varphi(t;\eta)\neq0$, when
$\eta \in L_{\cF_T}^2(\Omega;L^2(D))\setminus\{0\}$ and $t \in [0, T )$. From this, together with (\ref{eq:without u}) and (\ref{functional:J}), we find that for each $\eta \in L_{\cF_T}^2(\Omega;L^2(D))\setminus\{0\}$,
\begin{equation}
\label{ineq:J}
\begin{array}{lll}
&&\cJ(\varphi(\cdot;\eta))\\[3mm]
&=&\dfrac{1}{2}\|\chi_G\varphi(\cdot;\eta)\|^2_{L_{\F}^2(0,T;L^2(D))}+\E\left\langle y_0,\varphi(0;\eta) \right\rangle\\[3mm]
&=&\dfrac{1}{2}\bigg\|\chi_G\varphi(\cdot;\eta)+\dfrac{\E\left\langle y(T;y_0,0),\eta \right\rangle}{\|\chi_G\varphi(\cdot;\eta)\|_{L_{\F}^2(0,T;L^2(D))}}\bigg\|^2_{L_{\F}^2(0,T;L^2(D))}\\[3mm]
&&-\dfrac{1}{2}\bigg[ \dfrac{\E\left\langle y(T;y_0,0),\eta \right\rangle}{\|\chi_G\varphi(\cdot;\eta)\|_{L_{\F}^2(0,T;L^2(D))}} \bigg]^2\\[3mm]
&\geq&-\dfrac{1}{2}\bigg[ \dfrac{\E\left\langle y(T;y_0,0),\eta \right\rangle}{\|\chi_G\varphi(\cdot;\eta)\|_{L_{\F}^2(0,T;L^2(D))}} \bigg]^2.
\end{array}
\end{equation}
Meanwhile, it follows from Lemma \ref{lemma:N-characteristic} and (\ref{eq:without u}) that
\begin{equation}
\label{relation:N-V-ineq1}
N(T,y_0)\geq \dfrac{\E\left\langle y(T;y_0,0),\eta \right\rangle}{\|\chi_G\varphi(\cdot;\eta)\|_{L_{\F}^2(0,T;L^2(D))}} \quad\ \text{for each}~\,\,\eta \in L_{\cF_T}^2(\Omega;L^2(D))\setminus\{0\},
\end{equation}
and
\begin{equation}
\label{relation:N-V-ineq2}
-N(T,y_0)=\inf_{\eta \in L_{\cF_T}^2(\Omega;L^2(D))\setminus\{0\}} \dfrac{\E\left\langle y(T;y_0,0),\eta \right\rangle}{\|\chi_G\varphi(\cdot;\eta)\|_{L_{\F}^2(0,T;L^2(D))}}
\leq \dfrac{\E\left\langle y(T;y_0,0),\eta \right\rangle}{\|\chi_G\varphi(\cdot;\eta)\|_{L_{\F}^2(0,T;L^2(D))}},
\end{equation}
for each $\eta \in L_{\cF_T}^2(\Omega;L^2(D))\setminus\{0\}$. Combining with (\ref{relation:N-V-ineq1}) and (\ref{relation:N-V-ineq2}), it follows that
$$
\bigg|\dfrac{\E\left\langle y(T;y_0,0),\eta \right\rangle}{\|\chi_G\varphi(\cdot;\eta)\|_{L_{\F}^2(0,T;L^2(D))}}\bigg|\leq N(T,y_0) \quad\ \text{for each}~\,\,\eta \in L_{\cF_T}^2(\Omega;L^2(D))\setminus\{0\},
$$
which implies
\begin{equation}
\label{relation:N-V-ineq3}
\sup_{\eta \in L_{\cF_T}^2(\Omega;L^2(D))\setminus\{0\}}\bigg\{ \bigg[ \dfrac{\E\left\langle y(T;y_0,0),\eta \right\rangle}{\|\chi_G\varphi(\cdot;\eta)\|_{L_{\F}^2(0,T;L^2(D))}} \bigg]^2 \bigg\}\leq N(T,y_0)^2.
\end{equation}
From (\ref{ineq:J}) and (\ref{relation:N-V-ineq3}), one can easily verify that
$
\cJ(\varphi)\geq \dfrac{1}{2}N(T,y_0)^2 \,\, \text{for each}~\,\,\eta \in L_{\cF_T}^2(\Omega;L^2(D))\setminus\{0\},
$
which shows
\begin{equation}
\label{relation:N-V-ineq4}
\inf_{\eta \in L_{\cF_T}^2(\Omega;L^2(D))\setminus\{0\}}\cJ(\varphi)\geq \dfrac{1}{2}N(T,y_0)^2.
\end{equation}
According to Remark \ref{remark:zero}, $0$ is not the minimizer of $\cJ(\cdot)$. This, along with the problem ${\bf(VP)}_{y_0}^T$, yields that
\begin{equation}
\label{relation:N-V-eq1}
V=\inf_{\eta \in L_{\cF_T}^2(\Omega;L^2(D))\setminus\{0\}}\cJ(\varphi)\quad\ \text{for each}~\,\,y_0 \in L_{\cF_0}^2(\Omega;L^2(D))\setminus\{0\}.
\end{equation}
Hence, from (\ref{relation:N-V-ineq4}) and (\ref{relation:N-V-eq1}), we are led to (\ref{relation:N-V-one side}).

Next, we prove that
\begin{equation}
\label{relation:N-V-another side}
V\leq-\dfrac{1}{2}N(T,y_0)^2.
\end{equation}
We first claim that
\begin{equation}
\label{ineq:N-0}
N(T,y_0)>0\quad\ \text{for each}~\,\,T\in \R^+.
\end{equation}
By contradiction, suppose we would have that
$
N(\hat{T},y_0)=0 \,\, \text{for some}~\,\, \hat{T}\in \R^+.
$
From the problem ${\bf(NP)}_{y_0}^{\hat{T}}$, it follows that $\hat{u}=0$ such that $y(\hat{T};y_0,0)=0$.
Consider the following system:
$$
\left\{
\begin{array}{ll}
dy=\triangle ydt+\chi_G\hat{u}dt+aydW(t), \quad\ &\mathrm{in} \quad\ D\times(0,\hat{T}),\\[3mm]
y=0, \quad\ &\mathrm{on} \quad\ \partial D\times(0,\hat{T}),\\[3mm]
y(0)=y_0,\quad\ &\mathrm{in} \quad\  D.
\end{array}
\right.
$$
By the above system, system (\ref{BSDE2}), and It\^{o}'s formula, we have
\begin{equation}
\label{eq:without u1}
\E\left\langle y(0;y_0,0),\varphi(0;\eta) \right\rangle=0\quad\ \text{for each}~\,\,\varphi(0;\eta) \in L_{\cF_0}^2(\Omega;L^2(D)).
\end{equation}
This yields $y(0;y_0,0)=y_0=0$, which leads to a contradiction, since we assumed that $y_0 \in L_{\cF_0}^2(\Omega;L^2(D))\setminus\{0\}$. Thus, (\ref{ineq:N-0}) is true.\\
By (\ref{eq:N-characteristic}) in Lemma (\ref{lemma:N-characteristic}), for any $\epsilon\in(0,N(T,y_0))$, there is a $\eta_\epsilon\in L_{\cF_T}^2(\Omega;L^2(D))$ such that
$$
\dfrac{\E\left\langle y_0,\varphi(0;\eta_\epsilon) \right\rangle}{\|\chi_G\varphi(\cdot;\eta_\epsilon)\|_{L_{\F}^2(0,T;L^2(D))}}\leq-N(T,y_0)+\epsilon.
$$
Then, together with the above inequality and (\ref{functional:J}) that for each $\lambda\geq0$,
$$
\begin{array}{lll}
\cJ(\varphi(\cdot;\lambda\eta_\epsilon))&=&\dfrac{1}{2}\lambda^2\|\chi_G\varphi(\cdot;\eta_\epsilon)\|^2_{L_{\F}^2(0,T;L^2(D))}+\lambda\E\left\langle y_0,\varphi(0;\eta_\epsilon) \right\rangle\\[3mm]
&\leq&\dfrac{1}{2}\bigg[\lambda\|\chi_G\varphi(\cdot;\eta_\epsilon)\|_{L_{\F}^2(0,T;L^2(D))}-(N(T,y_0)-\epsilon)\bigg]^2\\[3mm]
&&-\dfrac{1}{2}(N(T,y_0)-\epsilon)^2.
\end{array}
$$
By taking the infimum for $\lambda \in \R^+$ on both sides of the above inequality, we find that
$$
\inf_{\lambda\geq0}\cJ(\varphi(\cdot;\lambda\eta_\epsilon))\leq-\dfrac{1}{2}(N(T,y_0)-\epsilon)^2 \quad\ \text{for each}~\,\,\epsilon\in(0,N(T,y_0)).
$$
This, along with (\ref{relation:N-V-eq1}), it leads to
$$
V\leq-\dfrac{1}{2}(N(T,y_0)-\epsilon)^2\quad\ \text{for each}~\,\,\epsilon\in(0,N(T,y_0)).
$$
taking $\epsilon \rightarrow0$ in the above inequality yields (\ref{relation:N-V-another side}).

To sum up, (\ref{eq:relation-N-V}) follows from (\ref{relation:N-V-one side}) and (\ref{relation:N-V-another side}). This completes the proof.
\end{proof}

\subsection{Monotonicity and continuity of minimal norm function}

In order to prove our main results, we shall study the monotonicity of minimal norm function $N(\cdot,y_0)$ at first.  It is worth noting that we cannot simply mimic the calculations in the deterministic case by applying the time change technique, and treat the stochastic equations in the same way, since adaptedness is always required in the stochastic system. To overcome such technical difficulties, we adopt some methods
 from \cite{lv1} to prove the following properties of $N(\cdot, y_{0})$.
\begin{proposition}
\label{proposition:N-monotonicity}
Let $y_0 \in L_{\cF_0}^2(\Omega;L^2(D))\setminus\{0\}$. The following three assertions are true:\\[3mm]
$(a)$. The function $T \mapsto N(T , y_0)$ is strictly decreasing from $\R^+$ to $\R^+$;\\[3mm]
$(b)$. $\lim\limits_{T\rightarrow\infty}N(T , y_0)=0$;\\[3mm]
$(c)$. $\lim\limits_{T\downarrow0}N(T , y_0)=+\infty$.

\end{proposition}

\begin{proof}
At the beginning, according to (\ref{ineq:N-0}) and Proposition \ref{relation:N-V}, we have
$
N(T,y_0)\in \R^+ \,\,\text{for each}~\,\,T\in \R^+.
$
Then for any fixed $T_1, T_2 \in \R^+$ with $T_2 > T_1$, by Proposition \ref{relation:N-V} again, problem ${\bf(NP)}_{y_0}^{T_1}$ has an optimal
control $\hat{u}_1$ such that
\begin{equation}
\label{eq:N-optimal control}
y(T_1;y_0,\hat{u}_1)=0 \quad\ \text{and}~\quad\ \|\hat{u}_1\|_{L_{\F}^2(0,T_1;L^2(D))}= N(T_1, y_0).
\end{equation}
Define
$$
\bar {u}(t)=
\left\{
\begin{array}{ll}
\hat{u}_1(t),\quad\ &t\in(0,T_1),\\[3mm]
0,\quad\ &t\in(T_1,T_2).
\end{array}
\right.
$$
Then consider the following system with a control $\bar {u}(\cdot)\in L_{\F}^2(0,T_2;L^2(D))$:
$$
\left\{
\begin{array}{ll}
dy=\triangle ydt+\chi_G\bar {u} dt+aydW(t), \quad\ &\mathrm{in} \quad\ D\times(0,T_2),\\[3mm]
y=0, \quad\ &\mathrm{on} \quad\ \partial D\times(0,T_2),\\[3mm]
y(0)=y_0,\quad\ &\mathrm{in} \quad\  D.
\end{array}
\right.
$$
Along with the first equality of (\ref{eq:N-optimal control}), we have
$
y(T_2;y_0,\bar {u})=0$,
which means that $\bar {u}$ is a control to problem ${\bf(NP)}_{y_0}^{T_2}$.
By the optimality of $N(T_2, y_0)$ and the second equality of (\ref{eq:N-optimal control}), one has
$$
\begin{array}{lll}
N(T_2, y_0)\leq\|\bar {u}\|_{L_{\F}^2((0,T_2);L^2(D))}
=\|\hat {u}_1\|_{L_{\F}^2((0,T_1);L^2(D))}=N(T_1, y_0).
\end{array}
$$
Next, we claim that $N(T_2, y_0)<N(T_1, y_0)$. By contradiction, we suppose that it did not hold. Then by the above inequality, we would have $N(T_2, y_0)=N(T_1, y_0)$. Thus,
$$
\begin{array}{lll}
\|\bar {u}\|_{L_{\F}^2(0,T_2;L^2(D))}=\|\hat {u}_1\|_{L_{\F}^2(0,T_1;L^2(D))}=N(T_1, y_0)=N(T_2, y_0).
\end{array}
$$
This, together with $y(T_2;y_0,\bar {u})=0$, shows $\bar {u}$ is an optimal control to problem ${\bf(NP)}_{y_0}^{T_2}$. By the
bang-bang property of ${\bf(NP)}_{y_0}^{T_2}$ (see Theorem \ref{theorem-N-BB}, which will be given later), we have that $\|\chi_G\bar {u}(t)\|_{L^2_{\cF_t}(\Omega;L^2(D))}\neq0\, ,t \in(0,T_2 )$ a.e.  This contradicts the fact that $\bar {u} = 0$ over $(T_1, T_2)$. Hence, the assertion $(a)$ is true.

Next, according to (\ref{ineq:N-0}) and assertion $(a)$, one can easily verify that
$
\lim\limits_{T\rightarrow\infty}N(T , y_0)\\=0.
$
Hence, the assertion $(b)$ is true.

Finally, we prove the assertion $(c)$.
By contradiction, we suppose that it did not hold. Then there would be a sequence $\{T_n\}\subset(0,1)$ such that
$
\lim\limits_{T_n\downarrow 0}N(T_n , y_0)=\hat{N}\in\R^+.
$
Since $N(T , y_0)\in\R^+$ form (\ref{ineq:N-0}).
Let $u_n$ be the optimal control to problem $N(T_n,y_0)$ . Then from assertion $(a)$ and the above equality, we find that
\begin{equation}
\label{proposition:N-monotonicity-ineq2}
\|u_n\|_{L_{\F}^2(0,T_n;L^2(D))}= N(T_n,y_0)\leq \hat{N}.
\end{equation}
Consider the following system:
$$
\left\{
\begin{array}{ll}
dy=\triangle ydt+\chi_Gu_n dt+aydW(t), \quad\ &\mathrm{in} \quad\ D\times(0,T_n),\\[3mm]
y=0, \quad\ &\mathrm{on} \quad\ \partial D\times(0,T_n),\\[3mm]
y(0)=y_0,\quad\ &\mathrm{in} \quad\  D.
\end{array}
\right.
$$
From the above system, system (\ref{BSDE}), and It\^{o}'s formula, it holds
\begin{equation}
\label{proposition:N-monotonicity-eq3}
\E\left\langle y_0,z(0;\eta) \right\rangle+\E\int_0^{T_n}\left\langle u_n,\chi_Gz(t;\eta) \right\rangle dt=0,
\end{equation}
where $(z,Z)$ solves the equation (\ref{BSDE}).
Noting that (\ref{proposition:N-monotonicity-ineq2}) and $\E\|z\|^2\leq C \E\|\eta\|^2$,
$$
\bigg|\E\int_0^{T_n}\left\langle u_n,\chi_Gz(t;\eta) \right\rangle dt\bigg|\leq C\hat{N}\bigg(\int_0^{T_n}\big(\E\|\eta\|^2_{L^2(D)}\big)^2 dt\bigg)^{\frac{1}{4}}\sqrt {T_n}\rightarrow0\,\, \text{as}~\, T_n\downarrow 0.
$$
This, along with (\ref{proposition:N-monotonicity-eq3}), we have
$
\E\left\langle y_0,z(0;\eta) \right\rangle=0\,\, \text{for each}~\,\, z(0;\eta)\in L_{\cF_{0}}^2(\Omega;L^2(D)),
$
which implies
$
y_0=0.
$
This contradicts the assumption $y_0 \in L_{\cF_0}^2(\Omega;L^2(D))\setminus\{0\}$.
Hence, the assertion $(c)$ is true. The proof is complete.
\end{proof}

The following result is concerned with the continuity of minimal norm function $N(\cdot, y_0)$.

\begin{proposition}
\label{proposition:N-continuity}
Let $N(T , y_0)$ be defined in (\ref{problem:N}). Then, the function $N(\cdot, y_0)$ is continuous.

\end{proposition}

\begin{proof}
At the beginning, we show the right continuity of $N(\cdot, y_0)$.
Arbitrarily fix a $\hat{T} \in \R^+$. Let $\{T_n\} \subset(\hat{T},\hat{T}+1)$ be such that
\begin{equation}
\label{proposition:N-continuity-Tn}
T_n\downarrow \hat{T}.
\end{equation}
Then from the assertion $(a)$ of Proposition \ref{proposition:N-monotonicity}, there is a $\hat{M} \in \R^+$ such that
$$
\lim\limits_{T_n\downarrow\hat{T}}N(T_n , y_0)\\=\hat{M}.
$$
It suffices to show
\begin{equation}
\label{proposition:N-continuity-eq1}
N(\hat{T} , y_0)=\hat{M},
\end{equation}
which implies that the function $N(\cdot, y_0)$ is right continuity. i.e.,
\begin{equation}
\label{proposition:N-right-continuity}
\lim\limits_{T_n\downarrow\hat{T}}N(T_n , y_0)=N(\hat{T} , y_0).
\end{equation}
To this end, by contradiction, we suppose that (\ref{proposition:N-continuity-eq1}) did not hold. Then from the assertion $(a)$ of Proposition \ref{proposition:N-monotonicity}, we would have
\begin{equation}
\label{proposition:N-continuity-ineq-M}
\hat{M}< N(\hat{T} , y_0).
\end{equation}
Let $u_n$ be the optimal control to the problem ${\bf(NP)}_{y_0}^{T_n}$ and $y(\cdot;y_0,u_n)$ solve the following equation:
$$
\left\{
\begin{array}{ll}
dy_n=\triangle y_ndt+\chi_G u_ndt+ay_ndW(t), \quad\ &\mathrm{in} \quad\ D\times(0,T_n),\\[3mm]
y_n=0, \quad\ &\mathrm{on} \quad\ \partial D\times(0,T_n),\\[3mm]
y_n(0)=y_0,\quad\ &\mathrm{in} \quad\  D.
\end{array}
\right.
$$
Then we have
$$
\|u_n\|_{L_{\F}^2(0,T_n;L^2(D))}= N(T_n,y_0)\quad\ \text{and}~\quad\ y(T_n;y_0,u_n)=0.
$$
Define
$$
\hat{u}_n(t)=
\left\{
\begin{array}{ll}
u_n(t),\quad\ &t\in(0,T_n),\\[3mm]
0,\quad\ &t\in(T_n,\hat{T}+1).
\end{array}
\right.
$$
Consider the following equation:
\begin{equation}
\label{model8}
\left\{
\begin{array}{ll}
dy_n=\triangle y_ndt+\chi_G \hat{u}_ndt+ay_ndW(t), \quad\ &\mathrm{in} \quad\ D\times(0,\hat{T}+1),\\[3mm]
y_n=0, \quad\ &\mathrm{on} \quad\ \partial D\times(0,\hat{T}+1),\\[3mm]
y_n(0)=y_0,\quad\ &\mathrm{in} \quad\  D,
\end{array}
\right.
\end{equation}
which admits a solution $y_{n} = y(\cdot;y_0,\hat{u}_n)$. Similar to (\ref{ineq-y}), for a positive constant $C$ independent of $n$, we have
\begin{equation}
\label{ineq-y-2}
\|y_n\|_{L^2_{\mathbb{F}}(\Omega;C([0,T];L^2(D))) \cap L^2_\mathbb{F}(0,T;H_0^1(D))}
\leq C\big(\|y_0\|_{L^2_{\mathcal{F}_0}(\Omega;L^2(D))}+\|u_n\|_{L_{\F}^2(0,T;L^2(G))}\big).
\end{equation}
Then one can easily verify that
\begin{equation}
\label{proposition:N-continuity-ineq1}
\|\hat{u}_n\|_{L_{\F}^2(0,\hat{T}+1;L^2(D))}=\|u_n\|_{L_{\F}^2(0,T_n;L^2(D))}= N(T_n,y_0)\leq \hat{M},
\end{equation}
and
\begin{equation}
\label{proposition:N-continuity-eq2}
y(T_n;y_0,\hat{u}_n)=0.
\end{equation}
By (\ref{proposition:N-continuity-ineq1}), we know $\hat{u}_n$ is bounded in $L_{\F}^2(0,\hat{T}+1;L^2(D))$. Thus, we can extract a subsequence from $\{\hat{u}_n\}$, still denoted in the same way, such that for some $\hat{u}\in L_{\F}^2(0,\hat{T}+1;L^2(D))$,
\begin{equation}
\label{proposition:N-continuity-eq3}
{\hat{u}_n}\rightarrow\hat{u}\quad\ \text{weakly in}~\quad\ L_{\F}^2(0,\hat{T}+1;L^2(D)).
\end{equation}
This, combines with (\ref{proposition:N-continuity-ineq-M}) and (\ref{proposition:N-continuity-ineq1}), yields
\begin{equation}
\label{proposition:N-continuity-ineq2}
\|\hat{u}\|_{L_{\F}^2(0,\hat{T}+1;L^2(D))}\leq \liminf\limits_{n\rightarrow\infty}\|\hat{u}_n\|_{L_{\F}^2(0,\hat{T}+1;L^2(D))}\leq\hat{M}<N(\hat{T} , y_0).
\end{equation}
Meanwhile, from (\ref{proposition:N-continuity-eq3}), (\ref{proposition:N-continuity-Tn}) and (\ref{ineq-y-2}), one can directly verify that
\begin{equation}
\label{proposition:N-continuity-eq4}
y(T_n; y_0, \hat{u}_n)\rightarrow y(\hat{T}; y_0, \hat{u})\quad\ \text{weakly in}~\quad\ L^2_{\F}(\Omega;C([0,T];L^2(D))),
\end{equation}
where $y(\hat{T}; y_0, \hat{u})$ is a solution of the following system:
$$
\left\{
\begin{array}{ll}
dy=\triangle ydt+\chi_G\hat{u}dt+aydW(t), \quad\ &\mathrm{in} \quad\ D\times(0,\hat{T}+1),\\[3mm]
y=0, \quad\ &\mathrm{on} \quad\ \partial D\times(0,\hat{T}+1),\\[3mm]
y(0)=y_0,\quad\ &\mathrm{in} \quad\  D.
\end{array}
\right.
$$
These, along with (\ref{proposition:N-continuity-eq2}), lead to
$
y(\hat{T}; y_0, \hat{u})=0.
$
Thus, $\hat{u}$ is an admissible control to the problem ${\bf(NP)}_{y_0}^{\hat{T}}$ which yields
$
\|\hat{u}\|_{L_{\F}^2(0,\hat{T};L^2(D))}\geq N(\hat{T} , y_0).
$
This contradicts (\ref{proposition:N-continuity-ineq2}). Then (\ref{proposition:N-continuity-eq1}) is true. Hence, $N(\cdot , y_0)$ is right continuous over $\R^+$.

Next, we show that $N(\cdot , y_0)$ is left-continuous.
Let $\hat{T} \in \R^+$ again, and $\{T_n\}_{n\geq1}\subset(\frac{\hat{T}}{2},\hat{T})$ with $T_n\uparrow \hat{T}$.
We shall prove
\begin{equation}
\label{proposition:N-left-continuity}
\lim\limits_{T_n\uparrow\hat{T}}N(T_n , y_0)=N(\hat{T} , y_0).
\end{equation}
From (\ref{eq:N-characteristic}) in Lemma \ref{lemma:N-characteristic}, we have
\begin{equation}
\label{proposition:N-continuity-eq6}
N(T_n , y_0)=\sup_{\eta_n\in L_{\cF_{T_n}}^2(\Omega;L^2(D))\setminus\{0\}} \dfrac{\E\left\langle y_0,\varphi(0;\eta_n) \right\rangle}{\|\chi_G\varphi(\cdot;\eta_n)\|_{L_{\F}^2(0,T_n;L^2(D))}}.
\end{equation}
We choose $\eta_n\in L_{\cF_{T_n}}^2(\Omega;L^2(D))\setminus\{0\}$ such that
\begin{equation}
\label{proposition:N-continuity-eq7}
\int_0^{T_n}\E\|\chi_G\varphi(t;\eta_n)\|_{L^2(D)}^2dt=1.
\end{equation}
This, together with (\ref{proposition:N-continuity-eq6}), lead to
\begin{equation}
\label{proposition:N-continuity-eq8}
N(T_n , y_0)-\dfrac{1}{n}\leq \E\left\langle y_0,\varphi(0;\eta_n) \right\rangle.
\end{equation}
We first claim that for any $\{\eta_n\}\subseteq L_{\cF_{T_n}}^2(\Omega;L^2(D))\setminus\{0\}$,
\begin{equation}
\label{proposition:N-continuity-lim1}
\|\varphi(\cdot;\eta_n)-\varphi(\cdot;\eta_m)\|_Y\rightarrow0\quad\ \text{as}~\quad\ n,m\rightarrow\infty.
\end{equation}
Indeed, due to $Y$ is the completion of the space $X$ under the norm $\|\cdot\|_X$, and according to (\ref{norm:X}) and (\ref{eq:Y-norm}), it follows that for any $\{\eta_{n_k}\}\subseteq L_{\cF_{T_{n_k}}}^2(\Omega;L^2(D))\setminus\{0\}$,
$$
\|z(\cdot;\eta_{n_k})-z(\cdot;\eta_{m_k})\|_X=\|z(\cdot;\eta_{n_k})-z(\cdot;\eta_{m_k})\|_Y\rightarrow0\quad\ \text{as}~\quad\ k\rightarrow\infty.
$$
In other words,
\begin{equation}
\label{proposition:N-continuity-lim2}
\|\chi_Gz(\cdot;\eta_{n_k})-\chi_Gz(\cdot;\eta_{m_k})\|_{L_{\F}^2(0,T_{n_k};L^2(D))}\rightarrow0\quad\ \text{as}~\quad\ k\rightarrow\infty.
\end{equation}
By (\ref{space-Y}), we have
$$
\lim_{k\to \infty} \chi_G z(\cdot;\eta_{n_k})=\chi_G \varphi(\cdot,\eta_n).
$$
This, together with (\ref{proposition:N-continuity-lim2}), leads to
$$
\|\chi_G\varphi(\cdot;\eta_n)-\chi_G\varphi(\cdot;\eta_m)\|_{L_{\F}^2(0,T_n;L^2(D))}\rightarrow0\quad\ \text{as}~\quad\ n,m\rightarrow\infty,
$$
which yields (\ref{proposition:N-continuity-lim1}).
Hence, there exists a $\varphi(\cdot,\eta)\in Y$, such that
$$
\varphi(\cdot,\eta_n)\rightarrow\varphi(\cdot,\eta)\quad\ \text{strongly in}~\, Y.
$$
Combining with (\ref{proposition:N-continuity-eq7}), we obtain
\begin{equation}
\label{proposition:N-continuity-eq9}
\int_0^{\hat{T}}\E\|\chi_G\varphi(t;\eta)\|_{L^2(D)}^2dt=\|\chi_G\varphi(\cdot;\eta)\|_{L_{\F}^2(0,\hat{T};L^2(D))}=1.
\end{equation}
Meanwhile, from (\ref{proposition:N-continuity-eq8}), it follows that
\begin{equation}
\label{proposition:N-continuity-ineq3}
\limsup\limits_{n\rightarrow\infty} N(T_n , y_0)\leq \E\left\langle y_0,\varphi(0;\eta) \right\rangle.
\end{equation}
On the other hand, from (\ref{eq:N-characteristic}) in Lemma \ref{lemma:N-characteristic}, we obtain
$$
N(\hat{T} , y_0)=\sup_{\eta\in L_{\cF_T}^2(\Omega;L^2(D))\setminus\{0\}} \dfrac{\E\left\langle y_0,\varphi(0;\eta) \right\rangle}{\|\chi_G\varphi(\cdot;\eta)\|_{L_{\F}^2(0,\hat{T};L^2(D))}},
$$
which shows that for each $\eta\in L_{\cF_T}^2(\Omega;L^2(D))\setminus\{0\}$,
$$
N(\hat{T} , y_0)\|\chi_G\varphi(\cdot;\eta)\|_{L_{\F}^2(0,\hat{T};L^2(D))}\geq \E\left\langle y_0,\varphi(0;\eta) \right\rangle.
$$
This, together with (\ref{proposition:N-continuity-eq9}) and (\ref{proposition:N-continuity-ineq3}), lead to
\begin{equation}
\label{proposition:N-continuity-ineq4}
\limsup\limits_{n\rightarrow\infty} N(T_n , y_0)\leq N(\hat{T} , y_0).
\end{equation}
By the assertion $(a)$ in Proposition \ref{proposition:N-monotonicity}, it follows that
\begin{equation}
\label{proposition:N-continuity-ineq5}
\liminf\limits_{n\rightarrow\infty} N(T_n , y_0)\geq N(\hat{T} , y_0).
\end{equation}
Hence, from (\ref{proposition:N-continuity-ineq4}) and (\ref{proposition:N-continuity-ineq5}), we conclude that
$$
\lim\limits_{n\rightarrow\infty} N(T_n , y_0)= N(\hat{T} , y_0),
$$
which yields (\ref{proposition:N-left-continuity}). That is, $N(\cdot , y_0)$ is left-continuous.

The continuity of $N(\cdot , y_0)$ then follows from both the right continuity (\ref{proposition:N-right-continuity}) and the left continuity (\ref{proposition:N-left-continuity}) . The proof is completed.
\end{proof}

\section{Proof of main results}
In this section, we give proofs of our main results. i.e., Theorem \ref{th-equivalence-TM-N} and Theorem \ref{th1}. Our strategy is as follows. We first show that existences for norm optimal control problem ${\bf(NP)}_{y_0}^T$ and time optimal control problem ${\bf(TP)}_{y_0}^N$. Then the connections between time optimal  control problem ${\bf(TP)}_{y_0}^N$ and norm optimal  control problem ${\bf(NP)}_{y_0}^T$ are obtained, which will be used to prove Theorem \ref{th-equivalence-TM-N}. Finally, through utilizing the bang-bang property of norm optimal control problem ${\bf(NP)}_{y_0}^T$, we derive the bang-bang property for time optimal  control problem ${\bf(TP)}_{y_0}^N$. Furthermore, we verify Theorem \ref{th1}.

\subsection{Existences for norm optimal control problem and time optimal control problem}
At first, we present existences for norm optimal control problem ${\bf(NP)}_{y_0}^T$ and time optimal control problem ${\bf(TP)}_{y_0}^N$.
\begin{proposition}
\label{proposition:N-exist}
Assume $y_0 \in L_{\cF_0}^2(\Omega;L^2(D))\setminus\{0\}$. Then for any $T > 0$, the norm optimal control problem ${\bf(NP)}_{y_0}^T$ admits a solution, i.e., there exists at least one $u_*\in L_{\F}^2(0,T;L^2(D))$ such that
$
\|u_*\|_{L_{\F}^2(0,T;L^2(D))}= N(T, y_0),
$
and $y(T;y_0,u_*)=0$. Moreover, The control defined by
\begin{equation}
  \label{eq:u star}
  u_\ast = \chi_G \varphi^\ast,
\end{equation}
is the minimal norm optimal control, where $\varphi^\ast$ is the unique solution of the problem ${\bf(VP)}_{y_0}^T$ in (\ref{problem:V}).

\end{proposition}

\begin{proof}
By Lemma \ref{lemma:null controllable}, for system (\ref{model3}), there exists a control $v\in L_{\F}^2(0,T;L^2(D))$ so that
$
y(T;y_0,v)=0.
$
That is to say ${\bf(NP)}_{y_0}^T$ has admissible controls. So we can choose a sequence $\{ u_{*k} \}_{k\geq1}\subseteq L_{\F}^2(0,T;L^2(D))$ satisfying
\begin{equation}
\label{problem:N-exist-eq1}
y(T;y_0,u_{*k})=0\quad\ \text{for each}~\,k\in\N^+,
\end{equation}
and
\begin{equation}
\label{problem:N-exist-ineq1}
\|u_{*k}\|_{L_{\F}^2(0,T;L^2(D))}\leq N(T, y_0)+\dfrac{1}{k}\quad\ \text{for each}~\,k\in\N^+.
\end{equation}
By (\ref{problem:N-exist-ineq1}), there exists a subsequence of $\{ u_{*k} \}_{k\geq1}$, denoted in the same way, and a $u_*\in L_{\F}^2(0,T;L^2(D))$ such that
\begin{equation}
\label{problem:N-exist-lim1}
u_{*k}\rightarrow u_*\quad\ \text{weakly in}~\,L_{\F}^2(0,T;L^2(D))\quad\ \text{as}~\,k\rightarrow\infty.
\end{equation}
From this and using the inequality (\ref{problem:N-exist-ineq1}) again, it follows that
\begin{equation}
\label{problem:N-exist-ineq2}
\|u_*\|_{L_{\F}^2(0,T;L^2(D))}\leq\liminf\limits_{k\rightarrow\infty}\|u_{*k}\|_{L_{\F}^2(0,T;L^2(D))}\leq N(T, y_0).
\end{equation}
Meanwhile, by (\ref{problem:N-exist-lim1}), similar to the proof of (\ref{proposition:N-continuity-eq4}), we see that
\begin{equation}
\label{problem:N-exist-lim2}
y(T;y_0,u_{*k})\rightarrow y(T;y_0,u_*)\quad\ \text{weakly in}~\,L_{\F}^2(\Omega;C([0,T];L^2(D)))\quad\ \text{as}~\,k\rightarrow\infty.
\end{equation}
This, together with the equality (\ref{problem:N-exist-eq1}), implies that
$
y(T;y_0,u_*)=0.
$
i.e., $u_*$ is an admissible control to problem ${\bf(NP)}_{y_0}^T$.
Then by the optimality of $N(T, y_0)$ (see (\ref{problem:N})) and (\ref{problem:N-exist-ineq2}), we find that $u_*$ is an optimal control to
problem ${\bf(NP)}_{y_0}^T$.

Next, we will show that $u_\ast = \chi_G \varphi^\ast$ is the minimal norm optimal control, in the sense that
\begin{equation}
  \label{eq:7.8.1}
\|u_\ast\|_{L^2_{\mathbb{F}}(0,T;L^2(D))}\leq\|\hat{u}\|_{L^2_{\mathbb{F}}(0,T;L^2(D))},
\end{equation}
for any $\hat{u}\in L^2_\mathbb{F}(0,T;L^2(D))$ such that $y(T;y_0,\hat{u})
= 0$ in $D$, $\bP$-a.s.
According to the Lemma \ref{lemma:J}, $\varphi^\ast$ is the unique minimizer for the functional $\cJ(\varphi)$. By the optimality of $\varphi^\ast$, we obtain the
following Euler-Lagrange equation to the functional $\cJ(\varphi)$:
\begin{equation}
  \label{eq:EL}
  \int_0^T \E \left\langle \chi_G \varphi^\ast, \chi_G\psi\right\rangle dt + \E \left\langle y_0,\varphi(0)\right\rangle  =
  0, \ \text{for all}~ \psi\in Y.
\end{equation}
Using $u_\ast=\chi_G \varphi^\ast$ and plugging it into Euler-Lagrange equation (\ref{eq:EL}), we obtain
\begin{equation}
\label{eq:EL2}
  \int_0^T \E \left\langle u_\ast, \chi_G\psi\right\rangle dt +  \E \left\langle y_0,\varphi(0)\right\rangle  =
  0, \ \text{for all}~ \psi\in Y.
\end{equation}
Taking $\psi=\varphi$, by the system (\ref{model3}) with $u_\ast$, system (\ref{BSDE2}) and It\^{o}'s formula, it implies that
$y(T;y_0,u_\ast)
= 0$ in $D$, $\bP$-a.s.

Consider the following system:
$$
\left\{
\begin{array}{ll}
dy=\triangle ydt+\chi_G\hat{u} dt+aydW(t), \quad\ &\mathrm{in} \quad\ D\times(0,T),\\[3mm]
y=0, \quad\ &\mathrm{on} \quad\ \partial D\times(0,T),\\[3mm]
y(0)=y_0,\quad\ &\mathrm{in} \quad\  D,
\end{array}
\right.
$$
and by system (\ref{BSDE2}) and It\^{o}'s formula, we have
\[
\int_0^T \E \left\langle \hat{u}, \chi_G \varphi\right\rangle dt + \E \left\langle y_0, \varphi(0)\right\rangle =
  0, \ \text{for all}~ \varphi\in Y.
\]
This, together with equation \eqref{eq:EL2}, implies
\begin{equation}
  \label{eq:7.8.2}
\int_0^T \E \left\langle  u_\ast, \chi_G \varphi\right\rangle dt = \int_0^T \E \left\langle\hat{u}, \chi_G \varphi\right\rangle dt, \ \text{for all}~ \varphi\in Y.
\end{equation}
By the density argument, the equality \eqref{eq:7.8.2} still holds for
all $\psi\in Y$.
Thus, replacing $\varphi$ in \eqref{eq:7.8.2} by
$\varphi^\ast$ gives
$$
\int_0^T \E \|u_\ast\|^2_{L^2(D)} dt = \int_0^T \E \left\langle u^\ast,\hat{u}\right\rangle dt
\le\dfrac{1}{2}\int_0^T\E \|u_\ast\|^2_{L^2(D)}dt+\dfrac{1}{2}\int_0^T\E \|\hat{u}\|^2_{L^2(D)}dt.
$$
Therefore, the inequality \eqref{eq:7.8.1} is true. Hence, we end the proof of Proposition \ref{proposition:N-exist}.
\end{proof}

\begin{proposition}
\label{proposition:TN-exist}
Assume $y_0 \in L_{\cF_0}^2(\Omega;L^2(D))\setminus\{0\}$. Then for every $N > 0$, the time optimal control problem ${\bf(TP)}_{y_0}^N$ admits a solution, i.e., there exists at least one $u^*\in \cU_N$ such that
$
y(T (N, y_0);y_0,u^*)=0.
$
\end{proposition}

\begin{proof}
By Proposition \ref{proposition:N-monotonicity}, we can find $T \in\R^+$ such that,
\begin{equation}
\label{proposition:TN-exist-ineq1}
N(T,y_0)\leq N,
\end{equation}
where $N\in\R^+$ is a positive constant.
Then, according to Proposition \ref{proposition:N-exist}, there exists $u_*\in L_{\F}^2(0,T;L^2(D))$ such that
\begin{equation}
\label{proposition:TN-exist-eq1}
\|u_*\|_{L_{\F}^2(0,T;L^2(D))}= N(T, y_0)\quad\ \text{as well as}~\quad\ y(T;y_0,u_*)=0.
\end{equation}
Define
$$
\tilde{u}_*=
\left\{
\begin{array}{ll}
u_*,\quad\ &t\in(0,T),\\[3mm]
0,\quad\ &t\in(T,\infty).
\end{array}
\right.
$$
Then from (\ref{proposition:TN-exist-ineq1}) and (\ref{proposition:TN-exist-eq1}), we have
$
\tilde{u}_*\in \cU_N\,\,\text{as well as}~\,\, y(T;y_0,\tilde{u}_*)=0.
$
i.e., time optimal control problem ${\bf(TP)}_{y_0}^N$ has admissible controls.
So we can choose $\{ T_k \}_{k\geq1}\subseteq\R^+$ satisfying
\begin{equation}
\label{proposition:TN-exist-lim1}
T_k\rightarrow T(N,y_0)\quad\ \text{as}~\quad\ k\rightarrow 0,
\end{equation}
as well as $\{ u_k^* \}_{k\geq1}\subseteq L_{\F}^2(\R^+;L^2(D))$ holding
\begin{equation}
\label{proposition:TN-exist-eq2}
\|u_k^*\|_{L_{\F}^2(\R^+;L^2(D))}\leq N\quad\ \text{and}~\quad\ y(T_k;y_0,u_k^*)=0.
\end{equation}
By the inequality in (\ref{proposition:TN-exist-eq2}), there exists a subsequence of $\{ u_k^* \}_{k\geq1}$, denoted in the same manner, and
$u^*\in L_{\F}^2(\R^+;L^2(D))$ such that
\begin{equation}
\label{proposition:TN-exist-lim2}
u_k^*\rightarrow u^*\quad\ \text{weakly in}~\,L_{\F}^2(\R^+;L^2(D))\quad\ \text{as}~\,k\rightarrow\infty.
\end{equation}
and
$$
\|u^*\|_{L_{\F}^2(\R^+;L^2(D))}\leq\liminf\limits_{k\rightarrow\infty}\|u_k^*\|_{L_{\F}^2(\R^+;L^2(D))}\leq N,
$$
which indicates that
\begin{equation}
\label{proposition:TN-exist-eq3}
u^*\in \cU_N.
\end{equation}
Meanwhile, by (\ref{proposition:TN-exist-lim1}) and (\ref{proposition:TN-exist-lim2}), similar to the proof of (\ref{proposition:N-continuity-eq4}), we see that
\begin{equation}
\label{proposition:TN-exist-lim3}
y(T_k;y_0,u_k^*)\rightarrow y(T(N,y_0);y_0,u^*)\quad\ \text{weakly in}~\,L_{\F}^2(\Omega;C([0,T];L^2(D))).
\end{equation}
This, together with the equality in (\ref{proposition:TN-exist-eq2}), implies that
$
y(T(N,y_0);y_0,u^*)=0.
$
From the above equality and (\ref{proposition:TN-exist-eq3}), we conclude that $u^*$ is an optimal control to time optimal control problem ${\bf(TP)}_{y_0}^N$. The proof is complete.
\end{proof}
\subsection{The equivalence of minimal time and minimal norm controls}
In this subsection, we aim to discuss connections between time optimal  control problem ${\bf(TP)}_{y_0}^N$ and norm optimal  control problem ${\bf(NP)}_{y_0}^T$.

\begin{proposition}
\label{proposition:N0-T0}
Let $(N_0,T_0)\in\R^+\times\R^+$. Then, the following two conditions are equivalent:
\begin{itemize}
  \item [$(i)$] $N_0=N(T_0,y_0)$.
  \item [$(ii)$] $T_0=T(N_0, y_0)$.
\end{itemize}
\end{proposition}

\begin{proof}
This will be accomplished by two steps.

\underline{Step 1.} Relation $(i)$ implies relation $(ii)$.
On one hand, by the Proposition \ref{proposition:N-exist}, there exists an optimal norm control $u_{*1}\in L_{\F}^2(0,T_0;L^2(D))$ such that
$
\|u_{*1}\|_{L_{\F}^2(0,T_0;L^2(D))}\\= N(T_0, y_0),
$
and $y(T_0;y_0,u_{*1})=0$.
From $N_0 = N(T_0, y_0)$ and the optimality of $T (N_0, y_0)$, we have
\begin{equation}
\label{proposition:N0-T0-ineq1}
T (N_0, y_0)\leq T_0<+\infty.
\end{equation}
On the other hand, Since $N_0 > 0$, by the Proposition \ref{proposition:TN-exist}, there exists a time optimal  control $u^{*2}\in L_{\F}^2(0,T(N_0,y_0);L^2(D))$ such that
$$
y(T(N_0,y_0);y_0,u^{*2})=0\quad\ \text{and}~\quad\ \|u^{*2}\|_{L_{\F}^2(0,T(N_0,y_0);L^2(D))}\leq N_0.
$$
This, along with (\ref{problem:N}) leads to
$$
N(T(N_0,y_0),y_0)\leq \|u^{*2}\|_{L_{\F}^2(0,T(N_0,y_0);L^2(D))}\leq N_0.
$$
Since $N_0=N(T_0,y_0)$, together with the above inequality and the assertion $(a)$ in Proposition \ref{proposition:N-monotonicity}, it yields
\begin{equation}
\label{proposition:N0-T0-ineq2}
T (N_0, y_0)\geq T_0.
\end{equation}
From (\ref{proposition:N0-T0-ineq1}) and (\ref{proposition:N0-T0-ineq2}) we conclude that
$
T (N_0, y_0)= T_0,
$
as was to be proven.

\underline{Step 2.} Relation $(ii)$ implies relation $(i)$.
Since $T (N_0, y_0) = T_0\in \R^+$, by the Proposition \ref{proposition:TN-exist}, there exists a time optimal  control $u^{*1}\in L_{\F}^2(0,T_0;L^2(D))$ such that
$
y(T_0;y_0,u^{*1})=0\,\, \text{and}~\,\, \|u^{*1}\|_{L_{\F}^2(0,T_0;L^2(D))}\leq N_0.
$
Furthermore, from Proposition \ref{proposition:N-monotonicity}, we have
$
N_0>0\,\,\, \text{for each}~\,y_0 \in L_{\cF_0}^2(\Omega;L^2(D))\setminus\{0\},
$
which shows that there is a $\hat{T} \in \R^+$ such that
\begin{equation}
\label{proposition:N0-T0-eq1}
N_0=N(\hat{T},y_0).
\end{equation}
Using the proof procedure from Step 1 again, we obtain
$
\hat{T}=T(N_0,y_0).
$
Plugging it into (\ref{proposition:N0-T0-eq1}), it follows that
$
N_0=N(T(N_0,y_0),y_0).
$
Noting that $T(N_0,y_0)=T_0$, finally, we draw a conclusion
$
N_0=N(T_0,y_0).
$
Combining Steps 1 and 2, we complete the proof of the Proposition.
\end{proof}

We now are on the position to prove Theorem \ref{th-equivalence-TM-N}.

\begin{proof}[Proof of Theorem \ref{th-equivalence-TM-N}]
 Let $\Lambda$ be given by (\ref{equivalence}) and arbitrarily fix $(N, T ) \in\Lambda$ . Three facts are given in order.

First, by (\ref{equivalence}), Proposition \ref{proposition:N0-T0}, together with Proposition \ref{proposition:N-monotonicity} and Proposition \ref{proposition:N-continuity}, we have that
\begin{equation}
\label{th-equivalence-eq1}
N=N(T,y_0)\in\R^+\quad\ \text{and}~\quad\ T=T(N,y_0)>0.
\end{equation}
The above two equalities, along with Proposition \ref{proposition:N-exist} and Proposition \ref{proposition:TN-exist}, imply that both problem ${\bf(TP)}_{y_0}^N$ and problem ${\bf(NP)}_{y_0}^T$ have optimal controls.

Second, each optimal control $u^*$ to problem ${\bf(TP)}_{y_0}^N$ satisfies that
\begin{equation}
\label{th-equivalence-eq2}
y(T(N,y_0);y_0,u^*)=0\quad\ \text{and}~\quad\ \|u^*\|_{L^2_{\F}(\R^+;L^2(D))}\leq N.
\end{equation}
From (\ref{th-equivalence-eq1}) and (\ref{th-equivalence-eq2}), one can easily see that $u^*$ over $(0,T)$ is an optimal control to problem ${\bf(NP)}_{y_0}^T$.

Third, each optimal control $u_*$ to problem ${\bf(NP)}_{y_0}^T$ satisfies that
\begin{equation}
\label{th-equivalence-eq3}
y(T;y_0,u_*)=0\quad\ \text{and}~\quad\ \|u_*\|_{L^2_{\F}(0,T;L^2(D))}= N(T,y_0).
\end{equation}
Set
$$
\tilde{u}_*(t)=
\left\{
\begin{array}{ll}
u_*(t),\quad\ &t\in(0,T),\\[3mm]
0,\quad\ &t\in(T,+\infty).
\end{array}
\right.
$$
From (\ref{th-equivalence-eq1}) and (\ref{th-equivalence-eq3}), we see that $\tilde{u}_*(t)$ is an optimal control to problem ${\bf(TP)}_{y_0}^N$.

Finally, from the above three facts and Definition \ref{definition-equivalence-TM-N}, we see that problem ${\bf(TP)}_{y_0}^N$ and problem ${\bf(NP)}_{y_0}^T$ are equivalent. Hence, we complete the proof of Theorem \ref{th-equivalence-TM-N}.
\end{proof}

\subsection{Bang-bang property of time optimal controls}

In this subsection, we show that norm optimal  control problem ${\bf(NP)}_{y_0}^T$ has the bang-bang property. Then we derive the
bang-bang property for time optimal  control problem ${\bf(TP)}_{y_0}^N$ by virtue of the bang-bang property of problem ${\bf(NP)}_{y_0}^T$.

\begin{theorem}
\label{theorem-N-BB}
Assume $y_0 \in L_{\cF_0}^2(\Omega;L^2(D))\setminus\{0\}$. The norm optimal  control of problem ${\bf(NP)}_{y_0}^T$ satisfies the bang-bang property, i.e., the norm optimal  control $u_* \in L^2_{\F}(0,T;L^2(D))$ for problem ${\bf(NP)}_{y_0}^T$ satisfies
\begin{equation}
\label{theorem-N-BB-1}
\|\chi_Gu_*\|_{L_{\F}^2(0,T; L^2(D))}=N(T,y_0)
\end{equation}
and
\begin{equation}
\label{theorem-N-BB-2}
\|\chi_Gu_*(t)\|_{L^2_{\cF_t}(\Omega;L^2(D))}\neq0\quad\ \text{for}~\,t \in(0,T )\,\text{a.e.}
\end{equation}
\end{theorem}

\begin{proof}
  The equality in (\ref{theorem-N-BB-1}) is obvious due to Proposition \ref{proposition:N-exist}.
Then from (\ref{eq:u star}) in Proposition \ref{proposition:N-exist}, together with Remark \ref{remark:zero} and Remark \ref{remark:not zero}, finally, along with (\ref{theorem-N-BB-1}), one can easily verify that (\ref{theorem-N-BB-2}) is true. This completes the proof of the Theorem.
\end{proof}

\begin{theorem}
\label{theorem-TN-BB}
Assume $y_0 \in L_{\cF_0}^2(\Omega;L^2(D))\setminus\{0\}$. The time optimal  control of problem ${\bf(TP)}_{y_0}^{N_0}$ satisfies the bang-bang property, i.e., the time optimal  control $u^* \in \cU_{N_0}$ for problem ${\bf(TP)}_{y_0}^{N_0}$ satisfies
\begin{equation}
\label{theorem-TN-BB-1}
\|\chi_Gu_*\|_{L_{\F}^2(0,T(N_0,y_0); L^2(D))}=N_0
\end{equation}
and
\begin{equation}
\label{theorem-TN-BB-2}
\|\chi_Gu_*(t)\|_{L^2_{\cF_t}(\Omega;L^2(D))}\neq0, \quad \text{for}~t \in(0,T(N_0,y_0) )~ \text{a.e.}
\end{equation}
\end{theorem}

\begin{proof}
According to Proposition \ref{proposition:TN-exist}, there exists a time optimal  control $u^*\in L_{\F}^2(0,T(N_0,y_0);L^2(D))$ such that
\begin{equation}
  \label{eq:less-than-N0}
\|\chi_Gu^*\|_{L_{\F}^2(0,T(N_0,y_0);L^2(D))}\leq\|u^*\|_{L_{\F}^2(0,T(N_0,y_0);L^2(D))}\leq N_0.
\end{equation}
Then, the solution $y^* := y(\cdot; y_0, u^*)$ of the following controlled stochastic system:
$$
\left\{
\begin{array}{ll}
dy^*=\triangle y^*dt+\chi_Gu^*dt+ay^*dW(t), \quad\ &\mathrm{in} \quad\ D\times(0,T(N_0,y_0)),\\[3mm]
y^*=0, \quad\ &\mathrm{on} \quad\ \partial D\times(0,T(N_0,y_0)),\\[3mm]
y^*(0)=y_0,\quad\ &\mathrm{in} \quad\  D,
\end{array}
\right.
$$
admits
$$
y(T(N_0,y_0); y_0, u^*)=0.
$$
Write
$$
\tilde{N}_0:=\|\chi_Gu^*\|_{L_{\F}^2(0,T(N_0,y_0);L^2(D))},
$$
and
\begin{equation}
\label{theorem-TN-BB-eq1}
T_0:=T(N_0,y_0).
\end{equation}
We first claim that
\begin{equation}
\label{theorem-TN-BB-eq2}
\tilde{N}_0=N_0.
\end{equation}
By contradiction, suppose that the above equality was not true. Then we would have that
$
\tilde{N}_0<N_0
$
by inequality \eqref{eq:less-than-N0}.
By the optimality of norm optimal  control problem ${\bf(NP)}_{y_0}^{T(N_0,y_0)}$, it follows that
$
\tilde{N}_0\ge N(T_0,y_0).
$
In addition, from Proposition \ref{proposition:N0-T0}, along with (\ref{theorem-TN-BB-eq1}), it implies that
$
N_0=N(T_0,y_0),
$
which leads to a contradiction. Thus, (\ref{theorem-TN-BB-eq2}) is true. i.e., (\ref{theorem-TN-BB-1}) is correct.

Next, we shall prove
\eqref{theorem-TN-BB-2}.
Indeed, from Proposition \ref{proposition:TN-exist}, there exists $u^*\in \cU_{N_0}$ and $T(N_0,y_0)$ be the time optimal  control and optimal time, respectively, for time optimal  control problem ${\bf(TP)}_{y_0}^{N_0}$. Then by Theorem \ref{theorem-N-BB} and Proposition \ref{proposition:N0-T0},
\begin{equation}
\label{theorem-TN-BB-eq3}
N_0 = N(T (N_0, y_0), y_0),
\end{equation}
is the optimal norm for norm optimal  control problem ${\bf(NP)}_{y_0}^{T_0}$ with $T_0 = T (N_0, y_0)$. Therefore, $u^*$ is a norm optimal  control for norm optimal  control problem ${\bf(NP)}_{y_0}^{T_0}$ with $T_0 = T (N_0, y_0)$ as well. Finally, by \eqref{theorem-N-BB-2} in Theorem \ref{theorem-N-BB}, then along with (\ref{theorem-TN-BB-eq1}) and (\ref{theorem-TN-BB-eq3}), it implies (\ref{theorem-TN-BB-2}) holds. The proof is complete.
\end{proof}

Then Theorem \ref{th1} is a direct consequence of Proposition \ref{proposition:TN-exist} and Theorem \ref{theorem-TN-BB}.

\section*{Acknowledgments}
This work is supported by the National Natural Science Foundation of China, the Science Technology Foundation of Hunan Province.

\bibliographystyle{abbrvnat}
\bibliography{ref.bib}

\begin{thebibliography}{35}
\providecommand{\natexlab}[1]{#1}
\providecommand{\url}[1]{\texttt{#1}}
\expandafter\ifx\csname urlstyle\endcsname\relax
  \providecommand{\doi}[1]{doi: #1}\else
  \providecommand{\doi}{doi: \begingroup \urlstyle{rm}\Url}\fi

\bibitem[Aoki(1961)]{aoki1961stochastic}
M.~Aoki.
\newblock Stochastic-time optimal-control systems.
\newblock \emph{Transactions of the American Institute of Electrical Engineers,
  Part II: Applications and Industry}, 80\penalty0 (2):\penalty0 41--46, 1961.

\bibitem[Balakrishnan(1965)]{balakrishnan1965optimal}
A.~Balakrishnan.
\newblock Optimal control problems in banach spaces.
\newblock \emph{Journal of the Society for Industrial and Applied Mathematics,
  Series A: Control}, 3\penalty0 (1):\penalty0 152--180, 1965.

\bibitem[Bellman et~al.(1956)Bellman, Glicksberg, and Gross]{bellman1956bang}
R.~Bellman, I.~Glicksberg, and O.~Gross.
\newblock On the “bang-bang” control problem.
\newblock \emph{Quarterly of Applied Mathematics}, 14\penalty0 (1):\penalty0
  11--18, 1956.

\bibitem[C{\^a}rja(1984)]{carja1984minimal}
O.~C{\^a}rja.
\newblock On the minimal time function for distributed control systems in
  banach spaces.
\newblock \emph{Journal of optimization theory and applications}, 44:\penalty0
  397--406, 1984.

\bibitem[C{\^a}rja(1993)]{carj1993minimal}
O.~C{\^a}rja.
\newblock The minimal time function in infinite dimensions.
\newblock \emph{SIAM journal on control and optimization}, 31\penalty0
  (5):\penalty0 1103--1114, 1993.

\bibitem[Chen et~al.(2018)Chen, Wang, and Yang]{Yang2018}
N.~Chen, Y.~Wang, and D.~Yang.
\newblock Time-varying bang--bang property of time optimal controls for heat
  equation and its application.
\newblock \emph{Systems \& Control Letters}, 112:\penalty0 18--23, 2018.

\bibitem[Durga et~al.(2021)Durga, Muthukumar, and Fu]{durga2021stochastic}
N.~Durga, P.~Muthukumar, and X.~Fu.
\newblock Stochastic time-optimal control for time-fractional ginzburg--landau
  equation with mixed fractional brownian motion.
\newblock \emph{Stochastic Analysis and Applications}, 2021.

\bibitem[Fattorini(1964)]{fattorini1964time}
H.~O. Fattorini.
\newblock Time-optimal control of solutions of operational differenital
  equations.
\newblock \emph{Journal of the Society for Industrial and Applied Mathematics,
  Series A: Control}, 2\penalty0 (1):\penalty0 54--59, 1964.

\bibitem[Fattorini(2005)]{fattorini2005infinite}
H.~O. Fattorini.
\newblock \emph{Infinite dimensional linear control systems: the time optimal
  and norm optimal problems}.
\newblock Elsevier, 2005.

\bibitem[Fattorini(2011)]{Fattorini2011}
H.~O. Fattorini.
\newblock Time and norm optimal controls: a survey of recent results and open
  problems.
\newblock \emph{Acta Mathematica Scientia}, 31\penalty0 (6):\penalty0
  2203--2218, 2011.

\bibitem[Gozzi and Loreti(1999)]{gozzi1999regularity}
F.~Gozzi and P.~Loreti.
\newblock Regularity of the minimum time function and minimum energy problems:
  the linear case.
\newblock \emph{SIAM Journal on Control and Optimization}, 37\penalty0
  (4):\penalty0 1195--1221, 1999.

\bibitem[Liu(2014)]{Liu2014}
X.~Liu.
\newblock Controllability of some coupled stochastic parabolic systems with
  fractional order spatial differential operators by one control in the drift.
\newblock \emph{SIAM Journal on Control and Optimization}, 52\penalty0
  (2):\penalty0 836--860, 2014.

\bibitem[Loh{\'e}ac and Tucsnak(2013)]{loheac2013maximum}
J.~Loh{\'e}ac and M.~Tucsnak.
\newblock Maximum principle and bang-bang property of time optimal controls for
  schrodinger-type systems.
\newblock \emph{SIAM Journal on Control and Optimization}, 51\penalty0
  (5):\penalty0 4016--4038, 2013.

\bibitem[L{\"u}(2011)]{Lu2011}
Q.~L{\"u}.
\newblock Some results on the controllability of forward stochastic heat
  equations with control on the drift.
\newblock \emph{Journal of Functional Analysis}, 260\penalty0 (3):\penalty0
  832--851, 2011.

\bibitem[L{\"u} and Zhang(2021)]{lv1}
Q.~L{\"u} and X.~Zhang.
\newblock \emph{Mathematical control theory for stochastic partial differential
  equations}.
\newblock Springer, 2021.

\bibitem[Micu et~al.(2012)Micu, Roventa, and Tucsnak]{micu2012time}
S.~Micu, I.~Roventa, and M.~Tucsnak.
\newblock Time optimal boundary controls for the heat equation.
\newblock \emph{Journal of Functional Analysis}, 263\penalty0 (1):\penalty0
  25--49, 2012.

\bibitem[Phung and Wang(2013)]{phung2013observability}
K.~D. Phung and G.~Wang.
\newblock An observability estimate for parabolic equations from a measurable
  set in time and its applications.
\newblock \emph{Journal of the European Mathematical Society}, 15\penalty0
  (2):\penalty0 681--703, 2013.

\bibitem[Proppe and Boyarsky(1977)]{proppe1977time}
H.~Proppe and A.~Boyarsky.
\newblock A time-optimal stochastic control problem.
\newblock \emph{International Journal of Systems Science}, 8\penalty0
  (10):\penalty0 1193--1199, 1977.

\bibitem[Tang and Zhang(2009)]{Tang2009}
S.~Tang and X.~Zhang.
\newblock Null controllability for forward and backward stochastic parabolic
  equations.
\newblock \emph{SIAM Journal on Control and Optimization}, 48\penalty0
  (4):\penalty0 2191--2216, 2009.

\bibitem[Tucsnak and Weiss(2009)]{tucsnak2009observation}
M.~Tucsnak and G.~Weiss.
\newblock \emph{Observation and control for operator semigroups}.
\newblock Springer Science \& Business Media, 2009.

\bibitem[Tucsnak et~al.(2016)Tucsnak, Wang, and Wu]{tucsnak2016perturbations}
M.~Tucsnak, G.~Wang, and C.-T. Wu.
\newblock Perturbations of time optimal control problems for a class of
  abstract parabolic systems.
\newblock \emph{SIAM Journal on Control and Optimization}, 54\penalty0
  (6):\penalty0 2965--2991, 2016.

\bibitem[Wang(2008)]{wang2008null}
G.~Wang.
\newblock L$^{\infty}$-null controllability for the heat equation and its
  consequences for the time optimal control problem.
\newblock \emph{SIAM journal on control and optimization}, 47\penalty0
  (4):\penalty0 1701--1720, 2008.

\bibitem[Wang and Xu(2013)]{wang2013equivalence}
G.~Wang and Y.~Xu.
\newblock Equivalence of three different kinds of optimal control problems for
  heat equations and its applications.
\newblock \emph{SIAM Journal on Control and Optimization}, 51\penalty0
  (2):\penalty0 848--880, 2013.

\bibitem[Wang and Zuazua(2012)]{Wang2012}
G.~Wang and E.~Zuazua.
\newblock On the equivalence of minimal time and minimal norm controls for
  internally controlled heat equations.
\newblock \emph{SIAM Journal on Control and Optimization}, 50\penalty0
  (5):\penalty0 2938--2958, 2012.

\bibitem[Wang et~al.(2015)Wang, Xu, and Zhang]{Wang2015}
G.~Wang, Y.~Xu, and Y.~Zhang.
\newblock Attainable subspaces and the bang-bang property of time optimal
  controls for heat equations.
\newblock \emph{SIAM Journal on Control and Optimization}, 53\penalty0
  (2):\penalty0 592--621, 2015.

\bibitem[Wang et~al.(2018)Wang, Wang, Xu, and Zhang]{Wang2018}
G.~Wang, L.~Wang, Y.~Xu, and Y.~Zhang.
\newblock \emph{Time optimal control of evolution equations}.
\newblock Springer, 2018.

\bibitem[Wang and Zhang(2015)]{wang2015norm}
Y.~Wang and C.~Zhang.
\newblock The norm optimal control problem for stochastic linear control
  systems.
\newblock \emph{ESAIM: Control, Optimisation and Calculus of Variations},
  21\penalty0 (2):\penalty0 399--413, 2015.

\bibitem[Wang et~al.(2016)Wang, Yang, Yong, and Yu]{wang2016exact}
Y.~Wang, D.~Yang, J.~Yong, and Z.~Yu.
\newblock Exact controllability of linear stochastic differential equations and
  related problems.
\newblock \emph{arXiv preprint arXiv:1603.07789}, 2016.

\bibitem[Yan(2021)]{yan2021time}
Z.~Yan.
\newblock Time optimal control of a clarke subdifferential type stochastic
  evolution inclusion in hilbert spaces.
\newblock \emph{Applied Mathematics \& Optimization}, pages 1--28, 2021.

\bibitem[Yang and Zhong(2016)]{Yang2016}
D.~Yang and J.~Zhong.
\newblock Observability inequality of backward stochastic heat equations for
  measurable sets and its applications.
\newblock \emph{SIAM Journal on Control and Optimization}, 54\penalty0
  (3):\penalty0 1157--1175, 2016.

\bibitem[Yang and Zhong(2018)]{Yang2017}
D.~Yang and J.~Zhong.
\newblock Optimal actuator location of the minimum norm controls for stochastic
  heat equations.
\newblock \emph{Mathematical Control and Related Fields}, 8\penalty0
  (3\&4):\penalty0 1081--1095, 2018.

\bibitem[Yang et~al.(2019)Yang, Guo, Gui, and Yang]{Yang2019}
D.~Yang, B.~Guo, W.~Gui, and C.~Yang.
\newblock The bang--bang property of time-varying optimal time control for null
  controllable heat equation.
\newblock \emph{Journal of Optimization Theory and Applications}, 182:\penalty0
  588--605, 2019.

\bibitem[Yong and Zhou(1999)]{yong1999stochastic}
J.~Yong and X.~Y. Zhou.
\newblock \emph{Stochastic controls: Hamiltonian systems and HJB equations},
  volume~43.
\newblock Springer Science \& Business Media, 1999.

\bibitem[Yu(2014)]{yu2014approximation}
H.~Yu.
\newblock Approximation of time optimal controls for heat equations with
  perturbations in the system potential.
\newblock \emph{SIAM Journal on Control and Optimization}, 52\penalty0
  (3):\penalty0 1663--1692, 2014.

\bibitem[Zuazua(2006)]{zuazua2006controllability}
E.~Zuazua.
\newblock \emph{Controllability of partial differential equations}.
\newblock PhD thesis, Optimization and Control, 2006.

\end{thebibliography}

\end{document}